\documentclass[a4paper,11pt]{amsart}
\usepackage{standalone}
\usepackage{amsfonts}
\usepackage{amsthm}
\usepackage{amssymb}
\usepackage{amsmath}
\usepackage{theoremref}
\usepackage{enumerate}
\usepackage{bbm}
\usepackage{bm}
\usepackage{pgfplots}
\usepgfplotslibrary{fillbetween}
\usetikzlibrary{intersections}
\usepackage{mathtools}
\usetikzlibrary{patterns}

\usepackage{amsmath,amssymb,latexsym}
\usepackage{amsthm}
\usepackage{bm}
\usepackage{overpic}
\usepackage{epstopdf}
\usepackage{color}
\theoremstyle{plain}
\newtheorem{theorem}{Theorem}[section]
\newtheorem{corollary}{Corollary}[section]
\newtheorem{proposition}{Proposition}[section]
\newtheorem{lemma}{Lemma}[section]
\theoremstyle{definition}

\theoremstyle{remark}
\newtheorem{remark}{Remark}[section]
\newtheorem{conjecture}{Conjecture}

\newcommand{\field}[1]{\mathbb{#1}}
\newcommand{\D}{\field{D}}
\newcommand{\R}{\field{R}}

\newcommand{\N}{\field{N}}
\newcommand{\C}{\field{C}}

\renewcommand{\P}{\field{P}}
\newcommand{\T}{\field{T}}

\renewcommand{\Re}{\mathop{\rm Re}}

\def\XXint#1#2#3{{\setbox0=\hbox{$#1{#2#3}{\int}$}
\vcenter{\hbox{$#2#3$}}\kern-.5\wd0}}

\author[R. Orive]{Ram\'on Orive}
\address{Universidad de la Laguna e IMAULL,  Avenida Astrof\'isico Francisco S\'anchez, s/n, Departamento de An\'alisis Matem\'atico \\ 38206 San Crist\'obal de La Laguna, Santa Cruz de Tenerife,  Spain} \email{rorive@ull.edu.es}

\author[J. S\'anchez-Lara]{Joaqu\'in S\'anchez-Lara}
\address{Departamento de Matem\'atica Aplicada, Universidad de Granada,  Granada,  Spain} \email{jslara@ugr.es}

\author[D. Seco]{Daniel Seco}
\address{Universidad de la Laguna e IMAULL,  Avenida Astrof\'isico Francisco S\'anchez, s/n, Departamento de An\'alisis Matem\'atico \\ 38206 San Crist\'obal de La Laguna, Santa Cruz de Tenerife,  Spain} \email{dsecofor@ull.edu.es}

\title{Optimal polynomial approximants and orthogonal polynomials on the unit circle. An electrostatic approach. }

\date{\today}
\thanks{The first and second authors are partially funded by PID2024-158075NB-I00, through the Generaci\'on de Conocimiento programme of Spanish Ministry of Science, Innovation and Universities; the third author is funded by grant PID2024-160185NB-I00,  through the Generaci\'on de Conocimiento programme and by grant RYC2021-034744-I of the Ram\'on y Cajal programme from Agencia Estatal de Investigaci\'on (Spanish Ministry of Science, Innovation and Universities).}

\begin{document}

\begin{abstract}
We explore the connection between two seemingly distant fields: the set of cyclic functions $f$ in a Hilbert space of analytic functions over the unit disc $\D$, on the one hand, and the families of orthogonal polynomials for a weight on the unit circle $\T$ (OPUC), on the other. This link is established by so-called Optimal Polynomial Approximants (OPA) to $1/f$, that is, polynomials $p_n$ minimizing the norm of $1-p_nf$, among all polynomials $p_n$ of degree up to a given $n$.

Here, we focus on the particular case of the Hardy space, and an electrostatic interpretation of the zeros of those OPA (and thus, of the corresponding OPUC) is studied. We find the electrostatic laws explaining the position of such zeros for a reduced but significant class of examples. This represents the first step towards a research plan proposed over a decade ago to understand zeros of OPA through their potential theoretic properties.

\end{abstract}
\maketitle

\section{Introduction}

In this article, we will explore a terrain at the intersection of operator theory and approximation theory, through a connection made in \cite{BKLSS2016} for the first time.  To better understand this, consider the \emph{shift} operator $S$ taking an analytic function $f$ to its multiplication by the independent variable.
A function $f$ is \emph{cyclic} in a space of analytic functions $H$ if $f$ belongs to no proper $S$-invariant subspace of $H$, that is, if $\P f$ is a dense subspace of $H$ (where $\P$ denotes the set of all polynomials). In relation to cyclic functions, a constructive approximation approach has been developed in the last decade from the work in \cite{BCLSS2015}, leading to the study of so-called optimal polynomial approximants. Our hope is to shed some light into the rationale behind the exact positions of the roots of such polynomials, following a path already hinted in that original article.

From now on, we also use the common notation that $\P_n$ stands for the subspace of polynomials of degree at most $n$. In most of the spaces studied, the constant function $1$ is cyclic and thus to determine whether a function $f$ is cyclic, it turns out to be enough to determine whether there exists a sequence of polynomials $\{q_n\}_{n\in\N}$ such that $1-q_nf$ converges to 0 in the norm of the space. We say that $p_n$ is an \emph{optimal polynomial approximant} (or \emph{OPA}) to $1/f$ of degree $n\in \N$ if $p$ minimizes the norm in $H$ of $1-pf$ for all $p \in \P_n$. Our aim is to better understand the zeros of OPA in terms of the choices of $f$ and $H$, although here we will concentrate on the case of $H$ being the simplest possible Hilbert space: The classical Hardy space, $H^2$, is the space of analytic functions over the unit disc with square summable Maclaurin coefficients. The norm can also be expressed in terms of the $L^2$ integral of the boundary values of the functions, which are well defined almost everywhere in a canonical sense. In the case that $H$ is $H^2$, Beurling's theorem provides a good characterization of cyclic functions as those that are \emph{outer}, that is, zero-free in the unit disc and such that the real part of its logarithm has the mean value property in the unit circle. We refer the reader to \cite{Garnett} for more on the basics of $H^2$, but what matters for us at this point is that this is a context in which cyclic functions are already well understood. This is in contrast to other function spaces such as the Dirichlet or Bergman spaces (see \cite{EFKMR2014} or \cite{HKZ2000} respectively), where important problems remain open. All these are examples of reproducing kernel Hilbert spaces over the unit disc $\D$, that is, spaces where, for each $\omega \in \D$, there exists an element $k_{\omega}$: $f(\omega) = \left\langle f, k_{\omega}\right\rangle$ for any $f$ in the space. We also recommend the classic \cite{BrSh} in which cyclicity is shown to require a necessary condition on logarithmic capacity. There, the celebrated Brown-Shields conjecture about cyclic functions is proposed regarding the sufficiency of their necessary condition \cite{BrSh}. An OPA-based attempt to attack the conjecture could consist on explaining the positions of the zeros of OPA in terms of potentials. If the potential appearing in the case of the Dirichlet space is a logarithmic potential plus an external field, then classical techniques from approximation theory (like the ones described in \cite{SaffTotik}) may include logarithmic capacity in a \emph{sufficient} condition for cyclicity. This is, as of now, still far away, and our plan is more modest: we focus on $H^2$ because we expect the potentials to be easier to identify. There, we are able to identify the role of two features of $f$ on the position of zeros of its corresponding OPA: the relative position of the zeros of $f$ and the multiplicity of said zeros. This plan was already hinted in Section 5 of the original article \cite{BCLSS2015}. Back indeed to $H^2$, OPAs connections with orthogonal polynomials on the unit circle (OPUC) were established in \cite{BKLSS2016} and in this paper we exploit such a connection in order to try and understand the position of the roots of the OPAs. We find that a mindset focused on electrostatic equilibria provides a good explanation for the position of the zeros.  We are also interested on the consequent asymptotic properties of their zeros (in terms of their degree).

Now consider a fixed $f\in H^2$ and the corresponding family of OPA $\{p_n\}_{n\in \N}$. The fact that $p_n f$ is the orthogonal projection of $1$ onto the subspace $f\,\P_n$ guarantees the existence and uniqueness of OPAs (provided $f \not \equiv 0$) and allows us to compute $p_n$: denote by $\varphi_k$ the polynomial of degree at most $k$ obtained from applying the Gram-Schmidt method to obtain an orthonormal basis $\{\varphi_k f\}_{k=0}^n$ of the space $f \P_n$. Denote by $\mu$ the measure over $\T$ with $d\mu(\theta):= |f(e^{i\theta})|^2 d\theta$. The $L^2(\T)$ representation of the $H^2$ norm implies $\{\varphi_k\}_{k\in \N}$ are an orthonormal basis of $\P$, equipped with the norm inherited as a subspace of $L^2(\mu)$. That is $\{\varphi_k\}$ are the OPUC for $\mu$. This leads to \cite[Proposition 3.1]{BKLSS2016}, which gives $p_n$ as
\begin{equation}\label{kernel}
p_n(z) = \,\overline{f(0)}\,\sum_{k=0}^n\, \overline{\varphi_k(0)}\,\varphi_k(z)\,=\, \overline{f(0)}\, K_n(z,0)\,,
\end{equation}
where $K_n(x,y)$ denotes the reproducing kernel of the space $\P_n$ at the point $y$ with the inherited norm ($k_y(x)$, in the notation introduced above for general RKHS). Accordingly, $p_nf$ is the reproducing kernel at $0$ in $f\P_n$ (as a subspace of $L^2(d\theta)$). Using the theory of OPUC (see e.g. \cite[Chapter 1]{Simon}, we can deduce from \eqref{kernel} the following expression, valid provided that $\varphi_n(0)\neq 0$ (otherwise $p_n=p_{n-1}$):
\begin{equation}\label{reversed}
p_n(z) = \overline{f(0)}\,A_n\,\varphi^*_n(z),
\end{equation}
where $A_n$ is the leading coefficient of $\varphi_n$
and $\varphi^*_n$ denotes, as usual, the reversed polynomial
$$\varphi^*_n(z) = z^n\,\overline{\varphi_n}(1/z)\,.$$
As an immediate consequence of \eqref{reversed} and using again the basic theory of OPUC (see e.g. \cite{Simon} as a general reference), we have that the zeros of $p_n$ lie outside the closed unit disk $\overline{\D}$.

With respect to the cyclic behaviour of a function $f$, some trivial cases include the following: firstly, when $f$ has at least one zero inside the disc, then every element of $f\P_n$ has that same zero and then the reproducing kernel property prevents $f$ from being cyclic; and on the other extreme, when $f$ is holomorphic beyond the boundary and zero-free on the closed disc, then $1/f$ is a really good function and its Taylor polynomials $q_n$ already make $1-q_nf$ converge to 0, making $f$ cyclic. Thus, an interesting class is formed by what we may call \emph{critical polynomials}, that is, polynomial functions $f$ without zeros in $\D$ but with some zeros on $\T$. For these, the explicit computation of the OPA has been achieved in a large class of spaces (with definitions analogous to the one presented here, except for the choice of norm to be minimized): In \cite{BCLSS2015}, OPA to $1/f$ with $f(z) = 1-z$ are explicitly computed. In \cite{BKLSS2016} this explicit expression for the OPA is given for $f(z) = (1-z)^N,\;N\in \N\setminus \{0\}$, and this is finally extended to $f(z)=(1-z)^a\,,\; \Re a>0$ in \cite{BKLSS2019}. For polynomials with zeros at several points, see \cite{AcSe}. It is our belief that in order to understand the role that logarithmic capacity plays on the cyclicity in the Dirichlet space we must first describe the positions of the zeros of OPAs. To do so for any function and space seems out of the question for now. We will study here the positions of zeros of OPAs in $H^2$ to $1/f$, where $f$ is a critical polynomial (or a positive power of a critical polynomial). For the examples with a single zero of $f$, the zeros of the OPA seem to approach the unit circle $\T$ (from outside), but with a certain ``delay'' or gap around the critical point $z=1$ (as if some kind of ``repellent'' was located at this point). In addition, the strength of this repellent seems to increase with the multiplicity of the root (see \cite{BKLSS2016}). However, the asymptotic distribution of zeros (commonly called weak asymptotics) is still an open problem, although in \cite{AcSe} it was established that the zeros of $\{1-p_nf\}$ asymptotically follow the uniform distribution in the unit circle $\displaystyle d\mu(z) = \frac{d\theta}{2\pi}$. It is conjectured that the same should hold for the OPA $\{p_n\}$ themselves. In this sense, every point of the unit circle is known to be a limit point of the zeros of OPA, thanks to the Jentzsch-type theorem shown in \cite{BKLSS2019}.

In the present article, we provide a tool based on a system of non-hermitian orthogonality relations for the analysis of both OPUC and OPA, which in turn gives an electrostatic interpretation for the zeros of both families of polynomials, at least for the case of very simple functions $f$, all of them powers of critical polynomials. In Section \ref{sect2}, we introduce the notation and history relative to electrostatic partners that will be needed in order to correctly write some differential equations satisfied by the zeros of the OPUC. This is then exploited in Section \ref{sect3}, where we analyze some simple examples, starting with the case where $f(z)= (1-z)^a$, and finding an explanation to the exact position of the zeros of OPAs by showing that the zeros of the OPUC are in equilibrium with respect to an explicit non-classical Jacobi system of forces. Then, Section 4, the main section of the paper, is devoted to the case of generalized Jacobi type weights. In Theorem \ref{thm:gralized}, we provide a similar electrostatic view of the zeros of OPAs when $f$ is a critical polynomial with simple zeros, all of which are on $\T$, with real coefficients (i.e., the zeros of $f$ are symmetric with respect to the real axis). The particular case when $f$ has two conjugate simple zeros $e^{\pm i \theta}$ is then explored in further detail, and we analyze the presence of so-called \emph{spurious} zeros together with their dependence on $\theta$. When $\theta = \pi \frac{k}{l}$, with $k, l \in \N$, the asymptotic positions of spurious zeros of $p_n$ (and $\varphi_n$) are given a satisfactory explanation in terms of the congruences $n \mod l$. These (possible) spurious zeros of the OPUC (and the OPA) are studied in more detail in Section 5, and a well-supported conjecture is formulated.

In Section 6 we conclude by proposing some directions for further pursuing this line of research.

\section{Electrostatic interpretation of zeros of Orthogonal Polynomials}\label{sect2}

Throughout this section we recall some old and recent notions about electrostatic interpretation of zeros of polynomials.

The study of this topic starts with the seminal contribution of T. J. Stieltjes about the electrostatic interpretation of the zeros of Jacobi polynomials in the classical setting (see \cite{Stieltjes85}-\cite{Stieltjes85c}, and Szeg\H{o}'s book \cite{Szego} for a detailed explanation).
Indeed, Stieltjes posed the problem of finding the minimal energy configuration $-1\leq x_1\leq \ldots \leq x_n \leq 1$ of $n$ positive unit charges placed in the interval $[-1,1]$, in the presence of two positive charges $p,q>0$ located respectively at the endpoints $1$ and $-1$. A logarithmic interaction between the charges was assumed, in such a way that the energy of the system has the expression
\begin{equation}\label{energy}
E(x_1,\ldots,x_n) = \,\sum_{1\leq i<j \leq n}\,\log \left(\frac{1}{|x_i-x_j|}\right)  + p\,\sum_{i=1}^n\, \log \left(\frac{1}{|x_i-1|}\right) + q\,\sum_{i=1}^n\, \log \left(\frac{1}{|x_i+1|}\right).
\end{equation}

It is easy to see that the minimum of the energy \eqref{energy} is attained at an ``inner'' configuration $-1<x_1<x_2<\ldots <x_n<+1$. Hence, the absolute minimum is also a relative one. With this, Stieltjes was able to require that the gradient of this energy vanishes at this optimal configuration $-1<x^*_1<x^*_2<\ldots <x^*_n<+1$,  i.e.
\begin{equation}\label{gradient}
\frac{\partial E}{\partial x_k}\,(x^*_1,x^*_2,\ldots,x^*_n)  = \sum_{j \neq k}\,\frac{1}{x^*_k-x^*_j}\,+\,\frac{p}{x^*_k-1}\,+\,\frac{q}{x^*_k+1}\,= 0,\;k=1,\ldots,n.
\end{equation}
Observe that the left-hand side of \eqref{gradient} can be intepreted as the total force acting over the charge placed at $x_k$; the summation corresponds to the forces exerted by each charge at $x_j$ over the charge at $x_k$, while the remaining terms describe the forces that the charges at the endpoints $\pm 1$ exert over the charge at $x_k$. Thus, this conditions on the gradient means that the total force at each charge vanishes, and thus, that this distribution of charges is in equilibrium.

Denote by $y(x) = \prod_{k=1}^n\,(x-x^*_k)$ the monic polynomial whose zeros are the optimal locations of the charges.
Then, we can notice the identity
$$\frac{y''(x^*_k)}{2y'(x^*_k)}\,=\,\sum_{j \neq k}\,\frac{1}{x^*_k-x^*_j}.$$
From here, he readily showed that this polynomial $y$ satisfies the second order differential equation (see \cite[Ch. IV]{Szego})
\begin{equation}\label{jacobiODE}
y''(x) + \left(\frac{2p}{x-1} + \frac{2q}{x+1}\right)\,y'(x) - \frac{n(n+2p+2q-1)}{x^2-1}\,y(x) = 0\,.
\end{equation}
Then, comparing this equation with that satisfied by Jacobi polynomials, he concluded that $\displaystyle y = P_n^{(\alpha, \beta)}$, with $\alpha = 2p-1,\,\beta = 2q-1\,.$

Finally, it is not difficult to check that the corresponding Hessian matrix at this critical configuration is positive definite and, hence, that at $x_1^*,\ldots,x_n^*$ an absolute minimum of the logarithmic energy is attained (see e.g. \cite{VaVA}).

Stieltjes himself extended this interpretation to other classical orthogonal polynomials, like those of Laguerre and Hermite, as well as to those now called Heine-Stieltjes polynomials (see e.g. \cite{Szego}). During the more than one hundred years since the Stieltjes' contributions, some extensions of this model have appeared, including the case of prescribed negative (attractive) charges (and, in turn, the corresponding escape of the free charges to other regions of the complex plane, see e.g. the survey  \cite{MMM}). The identification of a family of polynomials with a sequence of such potentials whose gradient vanish is usually understood by \emph{electrostatic interpretation} of those polynomials.


Recently, in \cite{MOS2023}, a new tool to explain an electrostatic model for zeros of more general classes of polynomials satisfying both standard and non-standard orthogonality relations was introduced. Although this idea was initially conceived to provide an electrostatic interpretation for the zeros of multiple (or Hermite-Pad\'e) orthogonal polynomials, has indeed a very general applicability, as was pointed out in \cite{MOS2023}. Consider an integrable (possibly complex value) weight in some contour $\Delta$ in the complex plane and such that
\begin{equation}\label{semic}
\frac{w'}{w}=\frac{B}{A}\,,
\end{equation}
which means that $w$ is a semi-classical weight (using the common terminology in the theory of orthogonal polynomials). Indeed, the identity \eqref{semic} is a characterization of this concept.

 Now, consider given an arbitrary polynomial $p$ and a semi-classical weight $w$ on a contour $\Delta \subset \C$. The \textit{electrostatic partner} $S$ is another polynomial given by an expression of the form
\begin{equation} \label{EP}
	S :=c\, D_{w}[p]	 =\textcolor[rgb]{0.00,0.07,1.00}{c}\, \det
	\begin{pmatrix} p & 	\widehat p  \\ A   p' & A \left( \widehat p\right)' - B \widehat p  \end{pmatrix},
\end{equation}
where $\widehat{p}$ is given by $$\widehat p (x) = \int_{\Delta}\,\frac{p(t) w(t)}{t-x}\,d t\,, $$
usually called the \emph{second-kind function}, and $c$ is an adequate constant so that $S$ is a monic polynomial. . Let us point out that despite the seemingly strange appearance of $S$ in \eqref{EP}, it is very useful in many particular instances to get a nice electrostatic interpretation of the zeros of $p$ (see \cite[Sect. 8]{MOS2023}).


Indeed, in \cite{MOS2023} a second order differential equation for $y = p$ is found in terms of its electrostatic partner $S$, namely
\begin{equation}\label{SODE}
A S y'' + ((A'+B) S -A S') y'+ C y = 0\,,
\end{equation}
for another polynomial $C$, usually called a Van Vleck polynomial in the setting of polynomial solutions of Lam\'e type differential equations (see e.g. \cite{Marden}), among many other references to these topics).
This can be simplified to a Stieltjes-type model that follows for the zeros $z_1,\ldots,z_n$, provided they are simple, yielding
\begin{equation}\label{Smodel}
\left(\frac{y''}{2y'}\,+\,\left(\frac{A'}{A}+\frac{B}{A}-\frac{1}{2}\,\frac{S'}{S}\right)\,\right)(z_k) = 0,\,k=1,\ldots,n.
\end{equation}
Therefore, the zeros of $p$ are in equilibrium in the external field due to the weight $w$ (in terms of the zeros of $A$, the endpoints of the arcs and/or curves comprising the contour $\Delta$), as expected, and to the attraction of negative charges of magnitude $1/2$ at the zeros of the electrostatic partner $S$ (whose location is a priori unknown).

On the other hand, note that the previous construction is quite general: in principle, $p$ is an arbitrary polynomial, not necessarily orthogonal; however, when $p$ satisfies some kind of (non-hermitian) orthogonality relations (full orthogonality, quasi-orthogonality or, even, multiple orthogonality) we can say much more about the zeros of these electrostatic partners.
In particular, in \cite{MOS2023} it is established that if $p$ is the $n$-orthogonal polynomial with respect to the weight $w$ on a certain contour (not necessarily intervals of the real axis), satisfying \eqref{semic}, then its electrostatic partner $S=S_n$ (in principle, it depends on the degree $n$) has degree at most $\sigma$, the so-called class of the semi-classical weight, given by
$$\sigma = \max \{\deg A -2,\, \deg B -1\}\,.$$

\section{Electrostatics for OPA and OPUC. First examples.}\label{sect3}

After this brief review on electrostatic interpretation of the zeros of polynomials, we dive into the search for an electrostatic model for the zeros of OPA in the Hardy space $H^2$. From the identity \eqref{reversed}, the OPA to $1/f$ for a given function $f$ have their roots located at the conjugates of the inverses of the roots of OPUC with respect to the weight $|f(e^{i\theta})|^2d\theta$. Thus, electrostatic models for the roots of OPUC will directly provide models for the roots of OPA. We will also obtain information on the electrostatic behaviour of orthogonal polynomials on the Unit Circle for an important class of weights. Indeed, the electrostatic model for the roots of OPUC will be analyzed through the following subsections; its counterpart for the roots of OPA can be obtained using the following result.
\begin{lemma}
Suppose that the set of different points $\{z_1,\dots, z_n\}\in\mathbb{C}\setminus\{0\}$ satisfies
\begin{equation}\label{EqGen}
\sum_{j \neq k}\,\frac{1}{z_k-z_j}\,+\,\sum_{\ell=1}^N\frac{\lambda_\ell}{z_k-a_\ell}=0\,,\qquad k=1,\dots,n\,,
\end{equation}
with $\lambda_\ell\in\mathbb{R}$, $a_\ell\in\mathbb{C}$ and $N\in\mathbb{N}$, i.e., its distribution is in electrostatic equilibrium under the external field created by charges of value $\lambda_\ell$ placed at $a_\ell$, for $\ell =1,\ldots,N$. Consider the map $z\mapsto z^*=1/\overline{z}$.
\begin{enumerate}
\item If $a_\ell\neq 0$ for any $\ell\in\{1,\dots,N\}$, then the configuration $\{z_1^*,\dots,z_N^*\}$ satisfies
$$\sum_{j \neq k}\,\frac{1}{z_k^*-z_j^*}\,+\,\frac{-n+1-\sum_{\ell=1}^N\lambda_\ell}{z_k^*}\,+\,\sum_{\ell=1}^N\frac{\lambda_\ell}{z_k^*-a_\ell^*}=0\,,\qquad k=1,\dots,n\,,$$
that is, its distribution is in electrostatic equilibrium in the external field created by charges of value $\lambda_\ell$ at each $a_\ell^*$, plus an extra charge placed at $0$ of value $-n+1-\sum_{\ell=1}^N\lambda_\ell$.
\item If some $a_{\ell_0}=0$, then the configuration $\{z_1^*,\dots,z_n^*\}$ satisfies
$$\sum_{j \neq k}\,\frac{1}{z_k^*-z_j^*}\,+\,\frac{-n+1-\sum_{\ell=1}^N\lambda_\ell}{z_k^*}\,+\,\sum_{\ell=1,\ell\neq \ell_0}^N\frac{\lambda_\ell}{z_k^*-a_\ell^*}=0\,,\qquad k=1,\dots,n\,,$$
i. e., its distribution is in electrostatic equilibrium in the external field created by charges at $a_\ell^*$ of value $\lambda_\ell$ for $\ell\neq \ell_0$, and of value  $-n+1-\sum_{\ell=1}^N\lambda_\ell$ for the one at $0$.
\end{enumerate}
\end{lemma}
\begin{proof}
Suppose first that $a_\ell\neq 0$ for any $\ell\in\{1,\dots,N\}$. Then, \eqref{EqGen} is equivalent to
$$
0=\sum_{j \neq k}\,\frac{1}{1/z_k^*-1/z_j^*}\,+\,\sum_{\ell=1}^N\frac{\lambda_\ell}{1/z_k^*-1/a_\ell^*}=
\sum_{j \neq k}\,\frac{z_k^*z_j^*}{z_j^*-z_k^*}\,+\,\sum_{\ell=1}^N\frac{\lambda_\ell z_k^*a_\ell^*}{a_\ell^*-z_k^*}\,,\qquad k=1,\dots,n\,,$$
multiplying by $-1/(z_k^*)^2$ we obtain
$$
\sum_{j \neq k}\,\frac{z_j^*}{z_k^*(z_k^*-z_j^*)}\,+\,\sum_{\ell=1}^N\frac{\lambda_\ell a_\ell^*}{z_k^*(z_k^*-a_\ell^*)}=0\,,\qquad k=1,\dots,n\,,$$
and decomposing in simple fractions
$$
\sum_{j \neq k}\,\left(\frac{-1}{z_k^*}+\frac{1}{z_k^*-z_j^*}\right)\,
+\,\sum_{\ell=1}^N\left(\frac{-\lambda_\ell}{z_k^*}+\frac{\lambda_\ell}{z_k^*-a_\ell^*}\right)=0\,,\qquad k=1,\dots,n\,,$$
which proves the lemma when all the $a_j$'s are different from $0$.

The proof when some $a_{\ell_0}=0$, is the same but taking into account that in this case, if we apply the previous transforms to the term $\lambda_{\ell_0}/z_k$, we obtain the new term $-\lambda_{\ell_0}/z_k^*$.
\end{proof}

\subsection{If the function only has one zero.}

Now, let us start by considering the function
\begin{equation}\label{example1}
f(z) = (1-z)^a,\,a>0.
\end{equation}
After the very simple case where $a=1$ (considered in \cite{BCLSS2015}), this more general example was studied in \cite{BKLSS2019}, also dealing with the possibility that $a$ is a complex number with $\Re a>0$; but for the moment we restrict ourselves to the case where $a$ is a positive real number. We know that in this case the OPA $p_n$ agrees with the reversed orthogonal polynomial $\varphi^*_n$ (up to a multiplicative constant) with respect to the weight function
\begin{equation}\label{weight1}
w(z) = |(1-z)^a|^2\,d\theta,\,\theta \in [0,2\pi].
\end{equation}
The OPUC $\varphi_n$ with respect to this weight were studied in detail in \cite{IsWi} (see also \cite{Rang} for the case where $a$ is complex with $\Re a>0$). Indeed, in \cite{IsWi} the authors provided an expression of $\varphi_n$ in terms of hypergeometric functions \cite[(2.27)]{IsWi} and, as a consequence, they obtained the following second order differential equation for the corresponding Szeg\H{o} polynomials:
\cite[(2.36)]{IsWi}
\begin{equation}\label{ODEIsmail}
\varphi_n''(z) + \left(\frac{1-n-a}{z} + \frac{2a+1}{z-1}\right)\,\varphi_n'(x) + \frac{n(a+1)}{z(1-z)} \varphi_n(z) = 0\,.
\end{equation}
Comparing this ODE with \eqref{jacobiODE}, we easily get that this OPUC are also Jacobi polynomials (replacing the usual $z=-1$ by $z=0$), but for non-classical values of the parameters: $\alpha = 2a > 0,\,\beta = \beta_n = -n-a < -n$. Thus, the charge located at the origin is a large negative number, and so this point acts as an attractor for the free charges. Hence, all the charges cannot lie on the segment between $0$ and $1$. Now, most of the charges leave the real axis.

The previous arguments may be summarized in the following statement:

\begin{theorem}\label{thm:jacobi}
Let $f(z) = (1-z)^a,\,a>0$. Denote by $p_n$ the OPA to $1/f$, with $f$ given in \eqref{example1}, in the Hardy space $H^2$, of degree at most $n$. We have:
\begin{itemize}
\item[(a)] $p_n(z) = C_n \varphi^*_n(z)\,,$ with $C_n$ being a constant only depending on $n$ and $\varphi^*_n$ the reversed orthogonal polynomial with respect to the weight \eqref{weight1} in $\T$.
\item[(b)] The orthogonal polynomial $\varphi_n$ is a Jacobi polynomial, changing one of the ``endpoint'' locations from $-1$ to the origin, with non-classical values of the parameters, namely,
\begin{equation}\label{jacobi}
\varphi_n(z) = B_n\,P_n^{(\alpha,\beta)}\,,
\end{equation}
where $B_n$ is a constant only depending on $n$ and $$\beta = \beta_n = -n-a\,,\; \alpha = 2a\,. $$
\item[(c)] For each $n$, the zeros of $\varphi_n$ are in the equilibrium position of $n$ positive unit charges in the complex plane, in the presence of the external field created by a repellent (positive charge) of magnitude $a+1/2$ at $z=1$ and an attractor (negative charge) of magnitude $(n+a-1)/2$ placed at the origin (as usual, a logarithmic interaction is assumed between the charges).
\item[(d)] For each $n$, the zeros of $p_n$ are in the equilibrium position of $n$ positive unit charges in the complex plane, in the presence of the external field created by a repellent (positive charge) of magnitude $a+1/2$ at $z=1$ and an attractor (negative charge) of magnitude $(n+a)/2$ placed at the origin.

As consequence $p_n(z)=C_n\,P_n^{(\alpha,\beta)}(z)$ with
$$\alpha =2a \,,\; \beta =\beta_n= -n-a-1 \,, $$
where again this Jacobi polynomials has the ``endpoint'' location changed from $-1$ to $0$.
\end{itemize}

\end{theorem}

\begin{remark}\label{rem:equil}
It is necessary to point out that the equilibrium position mentioned in Theorem \ref{thm:jacobi} does not define a stable equilibrium configuration (absolute minimum of the logarithmic energy), as in the case of Jacobi polynomials with classical values of the parameters ($\alpha, \beta >-1$), and rather represents an unstable equilibrium (stationary or saddle point). Anyway, it is the unique type of equilibrium possible in this setting; observe in this sense that rewriting the ODE \eqref{ODEIsmail} in a Lam\'e type equation form (that is, multiplying by $z(1-z)$), the Van Vleck type polynomial $C$ (coefficient of $y$) is a constant. Then, the equilibrium configuration is unique.
\end{remark}


\begin{remark}\label{rem:jacobinonclassic}
As said, Theorem \ref{thm:jacobi} deals with Jacobi polynomials with non-standard values of the parameters. It is well known that in the classical setting $\alpha, \beta>-1$, Jacobi polynomials $P_n^{(\alpha,\beta)}$ are orthogonal in the real interval $(-1,1)$ with respect to the weight
$$w(x) = (1-x)^{\alpha} (1+x)^{\beta}\,.$$
This is, for $k=0,1,\ldots n$, the following nice system of orthogonality relations holds with some $c_n \neq 0$:
\begin{equation}\label{orthclas}
\int_{-1}^1\,x^k\, P_n^{(\alpha,\beta)}\,dx = c_n \delta_{n,k}\,.
\end{equation}
This is no longer valid as we leave the classical setting.

This is a subject that has been studied for quite some time. In \cite[Theorem 6.72]{Szego} the well-known Hilbert-Klein formulas give the precise number of zeros of the Jacobi polynomials (with arbitrary real parameters) in the real line (of course, it is necessary to adapt the conclusions to our setting, where the usual left endpoint $-1$ is replaced by $0$ . In the particular setting of Theorem \ref{thm:jacobi}, where $\alpha = -n-a < -n$ (since we assume that $a>0$) and $\beta = 2a+1 >1$ (the value of $\beta$ is indeed classical), we obtain that $\varphi_n$ has no real zeros when $n$ is even, and a single real zero in the interval $(-1,0)$ for $n$ odd. Therefore, the OPA $p_n$ (its reversed polynomial) has no real zero if n is even, and a single real zero in the interval $(-\infty,-1)$ for $n$ odd.
\end{remark}



$$w(x) = (1-x)^{\alpha} (1+x)^{\beta}\,.$$
In \cite{KMO2005}, the orthogonality relations for Jacobi polynomials with non-classical parameters are studied in detail. For general values of the parameters, many zeros of the polynomial leave the interval $(-1,1)$, even the real axis, and describe different type of trajectories in the complex plane. Since the orthogonality relations are non-hermitian, the contour of integration is not prescribed in advance since it may be analytically deformed. Moreover, for the different particular cases the Jacobi polynomials are completely characterized for the orthogonality relations on a single contour of the complex plane, in some cases, or for orthogonality conditions distributed between two or more contours in others: for multiple or Hermite-Padé orthogonality, see e.g. \cite{NiSo} as a classical monography, and for a recent contribution on this topic, see \cite{MOS2023} as well as the references therein. Also see \cite[Fig. 6.1]{KMO2005}. In the particular case considered in Theorem \ref{thm:jacobi}, with $\beta = \beta_n < -n$ and $\alpha > 1$, we have a single contour of orthogonality which surrounds the origin (the attractor) and contains $z=1$ (a weak repellent).

Moreover, since in this case we have a family of Jacobi polynomials $P_n^{(\alpha_n,\beta_n)}$ with varying parameters such that
$$\lim_{n\rightarrow \infty} \frac{\alpha_n}{n} = 0,\;\lim_{n\rightarrow \infty} \frac{\beta_n}{n} = -1\,,$$
it is not difficult to check that their zeros asymptotically follow the Lebesgue distribution on the unit circumference $\T$. This is the so-called weak asymptotics; see e.g. \cite{MFOr2005}, and references therein. 
In particular, since they are Szeg\H{o} polynomials, their zeros approach $\T$ from inside.
Therefore, bearing in mind that the zeros of our OPA are the reversed of those of the Szeg\H{o} polynomial, we easily conclude that their zeros approach the unit circle $\T$ from outside; this fact fits perfectly with what has been observed in the numerical examples (see e.g. \cite[Figs. 1-2]{BKLSS2016}).


\subsection{If the function has two opposite zeros.}

Our second example is a simple Jacobi-type function of the form
$$f(z) = (1-z)^a(1+z)^{b},\,a,b>-1/2\,,$$
which induces a weight of (hermitian) orthogonality
\begin{equation}\label{hermit2}
w(z)= |(1-z)^{2a}||(1+z)^{2b}|
\end{equation}
on $\T$.

Since $z\in \T$, we have $(1-\overline{z})= - \frac{1-z}{z}$ and this can be used to obtain the following equivalent expression for the weight $w$,
$$w(z) = e^{\pi i a} z^{-a-b} (1-z)^{2a} (1+z)^{2b}\,.$$ This way, the Szeg\H{o} polynomials $\varphi_n$, besides the system of orthogonality relations in $\T$ with respect to the weight \eqref{hermit2}, also satisfy a system of non-hermitian orthogonality relations $\varphi_n \perp \varphi_m$, for $m < n$, with respect to the varying (with respect to the degree $n$) weight
\begin{equation}\label{nonhermit2}
W_n(z) = z^{-n-a-b}\,(1-z)^{2a} (1+z)^{2b}\,.
\end{equation}
To avoid too much repetition, we provide a more general proof of this in Lemma \ref{lem:nonhermit} below. Going back to the discussion in Section \ref{sect2} on the differential equations and the electrostatic partner, the weight $W_n$ in \eqref{nonhermit2} is a semi-classical weight:
\[\frac{W'_n(z)}{W_n(z)}= -\frac{n+a+b}{z} - \frac{2a}{1-z} + \frac{2b}{1+z}.\]
In particular, it has class $\sigma=1$. We conclude that for each $n$ there exists a polynomial $S_n\in \P_1$ such that the linear differential equation \eqref{SODE} for $\varphi_n$ and the corresponding electrostatic model \eqref{Smodel} for its (simple) zeros $z_k,\,k=1,\ldots,n$, holds.
Thus, we have that,
$$\frac{A'(z)+B_n(z)}{A(z)}\,=\,\frac{1-n-a-b}{z} +\,\frac{2a+1}{z-1} +\,\frac{2b+1}{z+1}\,.$$
Denote by $s_n$ the unique zero of the electrostatic partner $S_n$ for each $n$, which by symmetry, should be real. Therefore, the zeros $z_j$ of the OPUC $\varphi_n$ (where $j=1,\ldots,n$) should satisfy the identity
$$\frac{\varphi''(z_j)}{2\varphi'(z_j)}\,+\,\left(\frac{1}{2}\,\frac{A'(z_j)+B_n(z_j)}{A(z_j)}\,-\,\frac{1}{2}\,\frac{1}{z_j-s_n}\right)\,=0\,.$$
The Szeg\H{o} polynomials $\varphi_n$ with respect to the weight \eqref{nonhermit2} are also considered in \cite{IsWi}. Take into account the analysis carried out in the paper by Ismail and Witte, in particular formulas (3.21)-(3.23). Then, in fact, $s_n$ does not depend on the particular degree $n$ but just on its parity:
$$s_n = s_0=\,\frac{b+a}{b-a}\;\;\text{for}\;n\;\;{even;}\;\;s_n=s_1=\frac{1}{s_0}=\,\frac{b-a}{b+a}\;\;\text{for}\;n\;\;\text{odd}.$$
Observe that if, for example, $b>a>0$, that is, if the repellent located at $-1$ is stronger than that in $+1$, the weak attractor of magnitude $1/2$ placed at $s=s_n$ is closer to $+1$ (from outside or inside, depending on the parity of the degree).

In Figure \ref{Fig2_1and-1} below, we can see some plots for the roots of Szeg\H{o} polynomials and OPA together with the roots of their respective electrostatic partners for $a=1$ and $b=3$. Observe that in these examples, when $n$ is odd, the root of the electrostatic partner is inside the circle and attracts along with it a root of the Szeg\H{o} polynomials $\varphi_n$ (these two roots are almost coincident for large $n$). In particular this means that, for odd $n$, there is one root of $p_n$ which does not approach the circle. Notice also that the electrostatic interpretation explains the different size of the gaps around $1$ and $-1$: the charge at $1$ has value $3/2$ while the charge at $-1$ has value $7/2$, which makes the gap around $-1$ bigger (leaving an approximately $3.5$ times larger empty space around $-1$).

\begin{figure}[h]
\includegraphics[width=7cm]{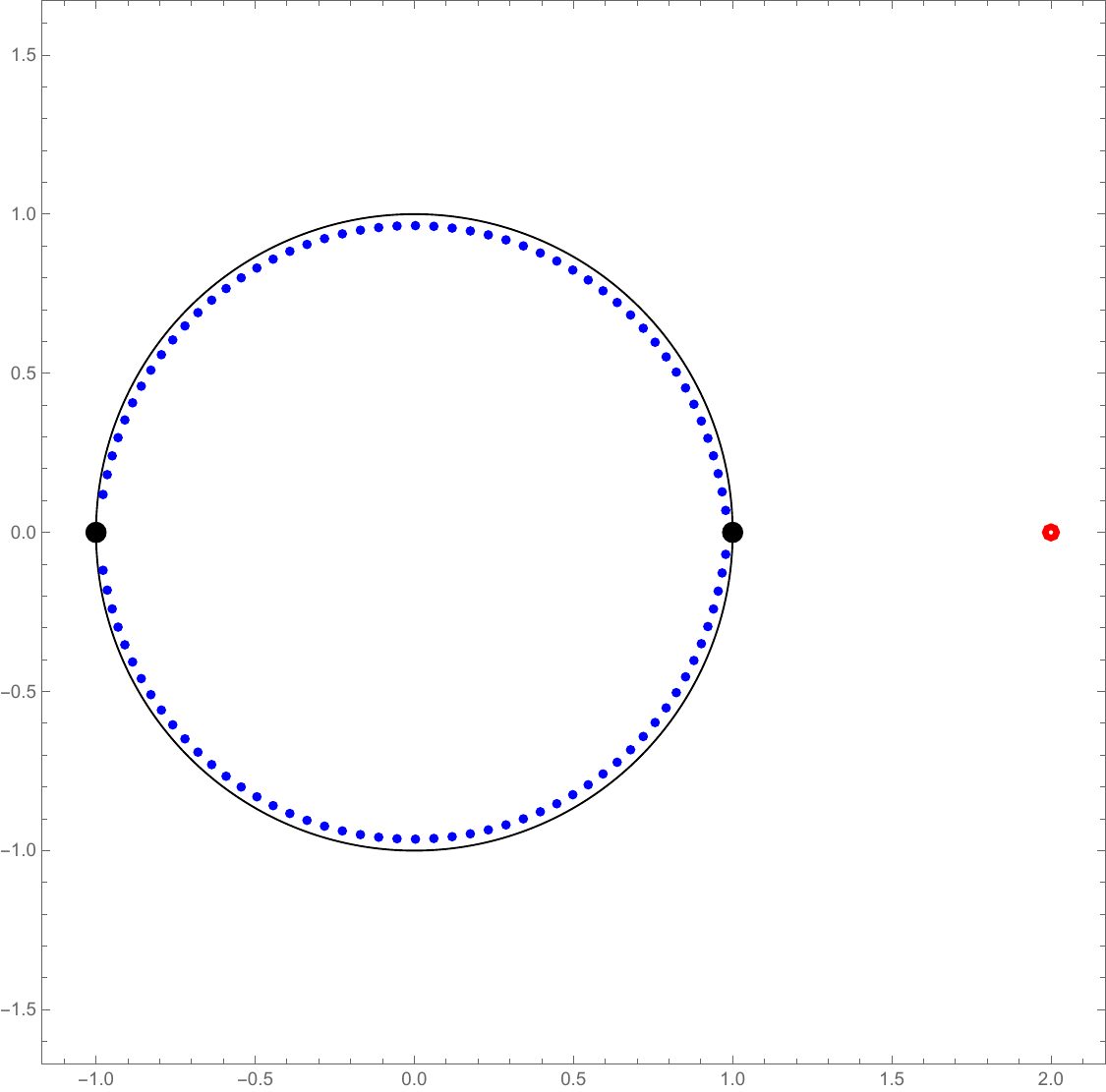}
\includegraphics[width=7cm]{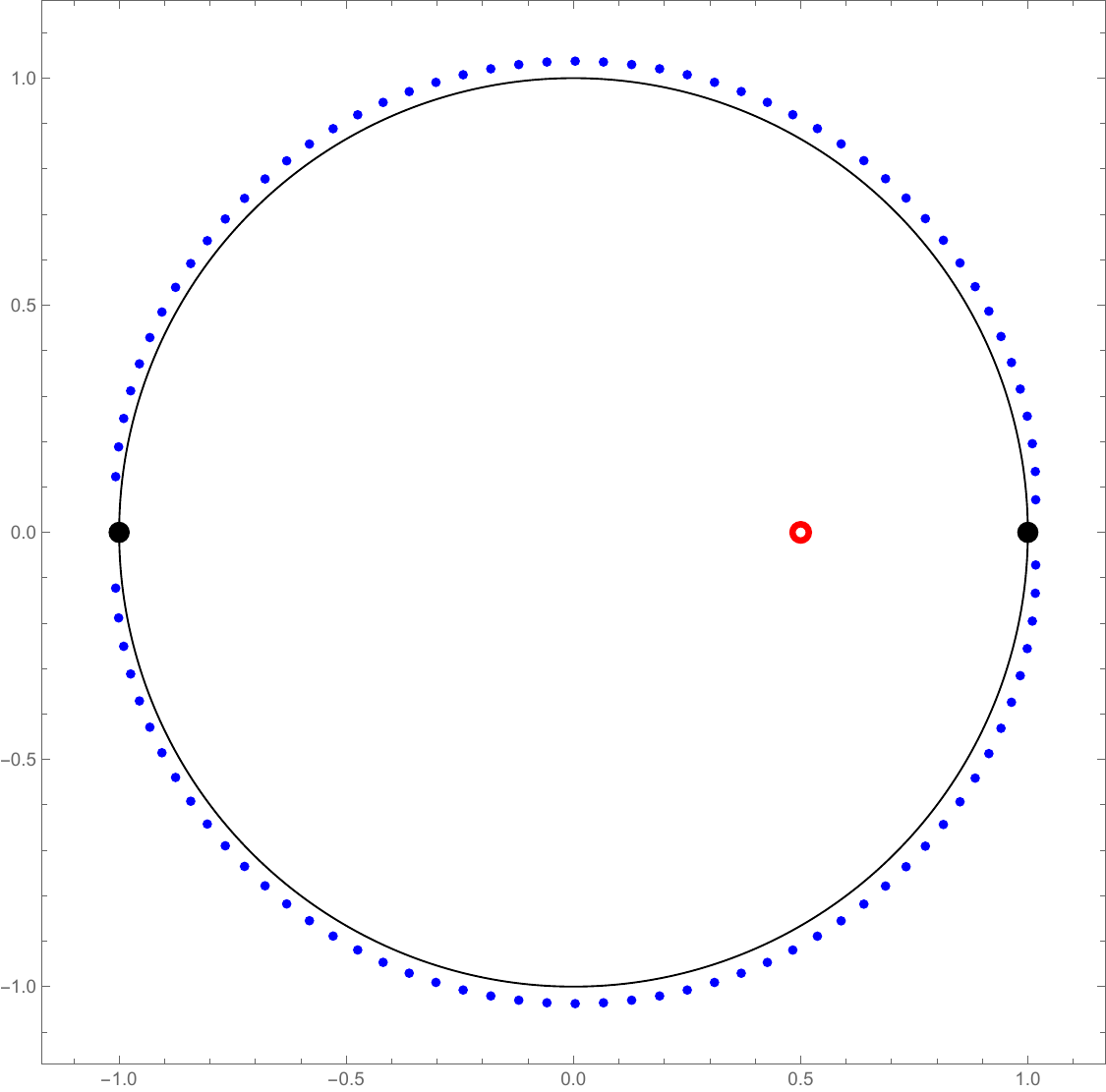}
\includegraphics[width=7cm]{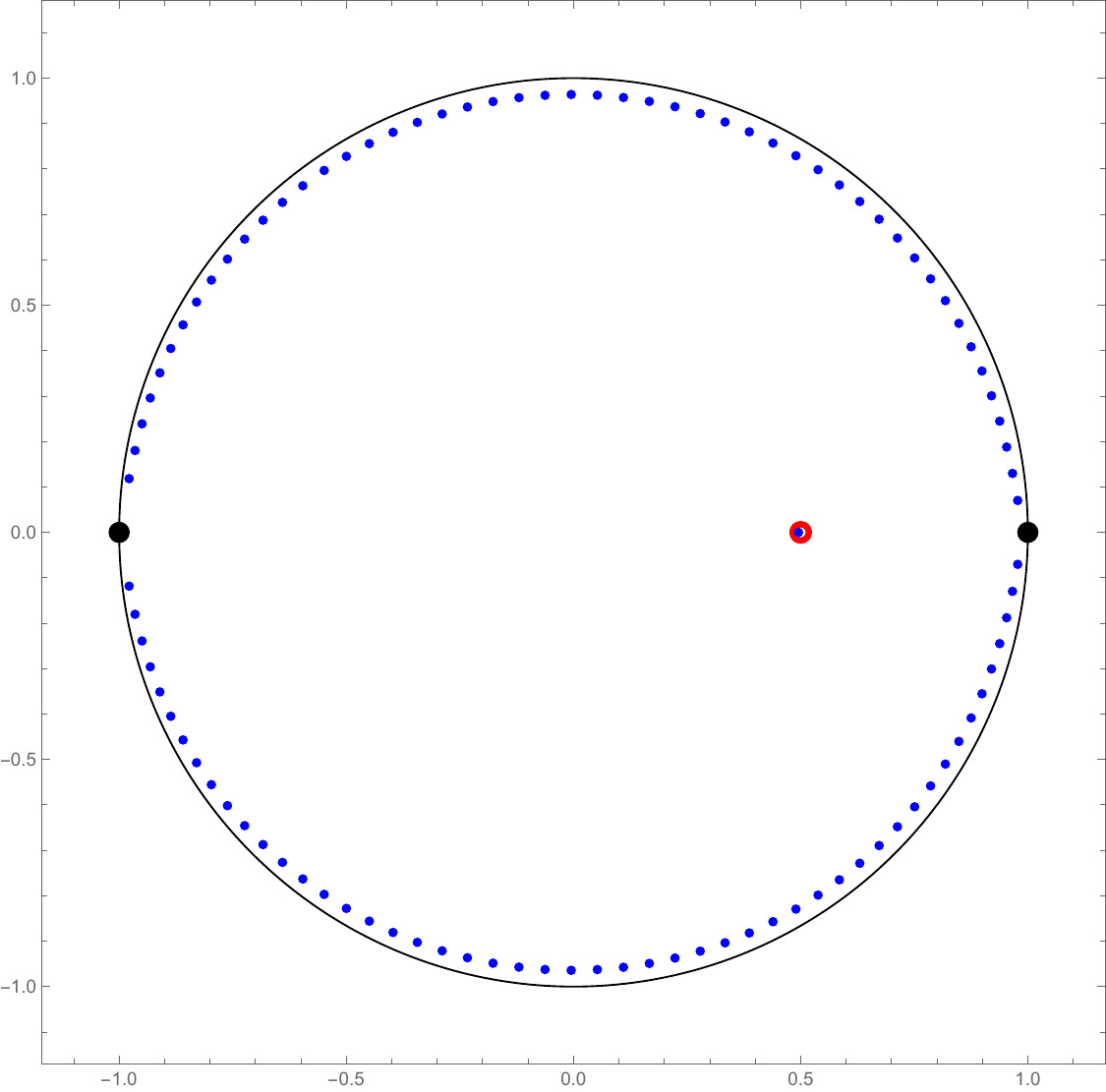}
\includegraphics[width=7cm]{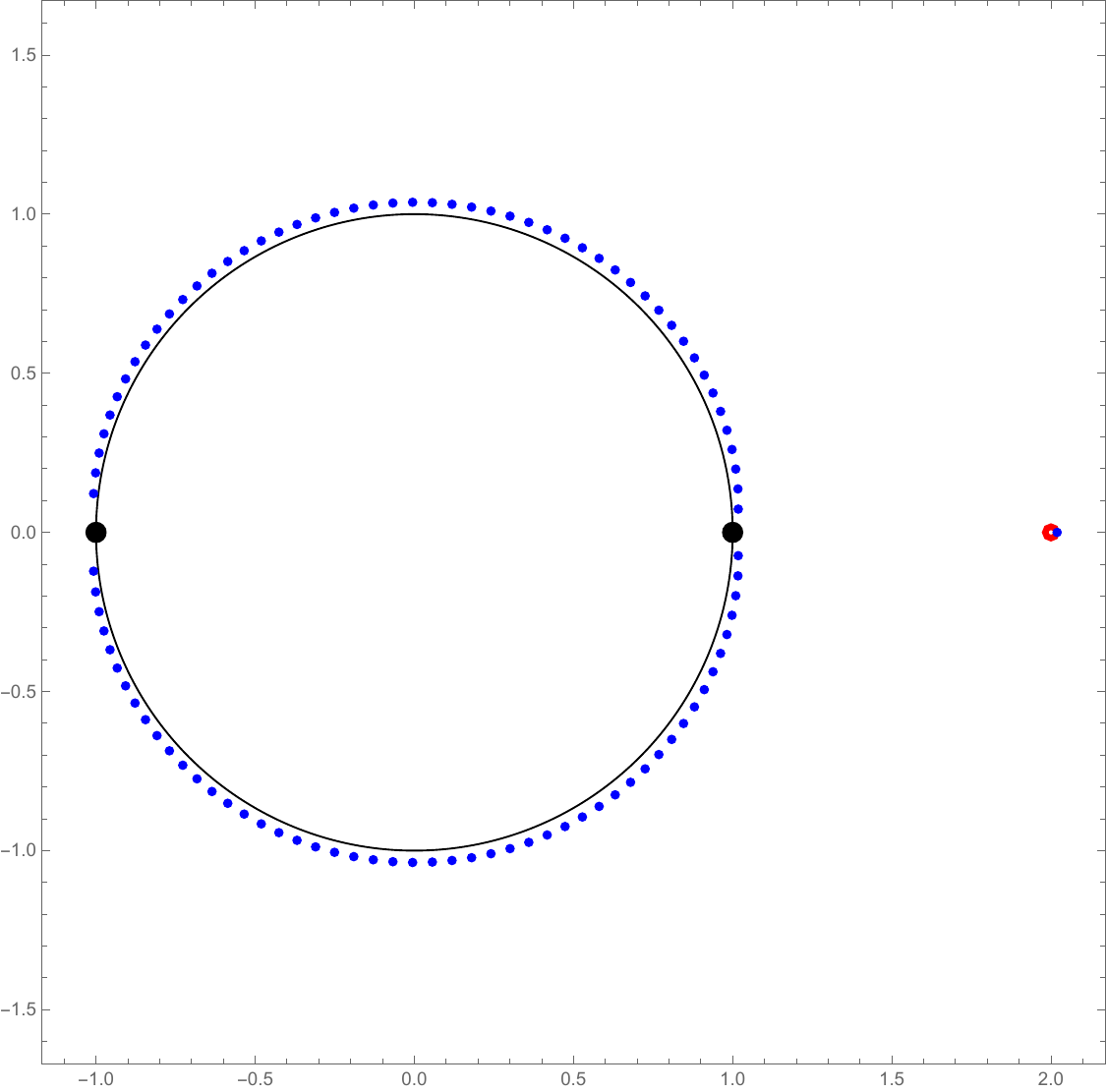}
\caption{\label{Fig2_1and-1} Roots (in blue) of $\varphi_n$ (left) and $p_n$ (right) for degrees $n=100$ (up) and $n=101$(down) associated to the weight $|1-z|^2\,|1+z|^6$ or, equivalently, to $f(z)=(1-z)(1+z)^3$. The spurious zeros or ``ghost'' roots correspond with the red circles.}
\end{figure}

\section{Electrostatics for OPA and OPUC with respect to Jacobi type weights.}

So far we have been studying the family of $\{\varphi_n\}$ for the measures \eqref{weight1} and \eqref{hermit2}; recall that the OPA $p_n$ agrees with the reversed polynomial up to a multiplicative constant, from the point of view of the Jacobi polynomials with non-standard parameters; however, we should recall that they are first and foremost OPUC. In this sense, in \cite{MMS2006} the authors studied OPUC with respect to a class of measures that includes both previous cases; namely, they considered the class of weights in $\T$ of the form
\begin{equation}\label{jacobiOPUC}
W(z) = w(z)\,\prod_{k=1}^m\,|z-a_k|^{2\beta_k}\,,
\end{equation}
where $|a_k|=1,\,\Re \beta_k >-1/2$ for all $k=1,\ldots,m$ and $w>0$ on $\T$ and can be extended as a holomorphic and non-vanishing function to an annulus containing the unit circle. In other words, they considered a class of generalized Jacobi-type weights on the unit circle. They use the powerful machinery of the Deift and Zhou's steepest descent method for Riemann-Hilbert problems (see \cite{DeZh} and \cite{Deift} as seminal references) to get strong asymptotics of the Szeg\H{o} polynomials $\varphi_n$, which are valid on the whole complex plane.
In particular, as for the zeros of $\varphi_n$, as $n\rightarrow \infty$, they prove that most of the zeros approach regularly and radially the unit circle, but close to each $a_k$ we have that the zeros converge faster to the unit circle than the rest, and leave a gap around $a_k$. 
Moreover, it could have at most $m-1$ spurious zeros that stay inside the unit circle. Notice that this is in agreement with our first example \eqref{example1}, where $W(z) = |(1-z)^a|^2$: the spurious zeros did not appear at all there, as predicted for $m=1$.

Following \cite{MMS2006}, let us consider now OPUC (Szeg\H{o} polynomials) with respect to weights of the form
\begin{equation}\label{GJOPUC}
w(z) = \prod_{j=1}^m\,|z-a_j|^{2\alpha_j},\; a_j\in \T,\,\alpha_j >-1/2, \,j=1,\ldots,m.
\end{equation}
First, observe that the (hermitian) orthogonality relations satisfied by this OPUC can be also interpreted as a (non-hermitian) system of orthogonality conditions on arcs  comprising the unit disk, which will be use to get an electrostatic model for its zeros. To see this, consider the complex function associated with the weight $w$ and $n$
$$W_n(z)= z^{-n}\,\prod_{j=1}^m\,z^{-\alpha_j}(z-a_j)^{2\alpha_j}\,,\qquad z\in\mathbb{C}\setminus\{a_1,\dots,a_m,0\}\,,$$
where the branch of each factor $(z-a_j)^{2\alpha_j}$ is taken to be positive for $z=a_j+t$ with $t>0$ and the cut goes from $a_j$ to $-\infty$ (if $\alpha_j\notin\mathbb{Z}$) following the segment from $a_j$ to $-1$ and then the interval $(-\infty,-1)$. The branch and the cut of the factor $z^{-\alpha_j}$ is defined as usual (positive in $(0,+\infty)$ and with the cut in $(-\infty,0)$).
\begin{lemma}\label{lem:nonhermit}
The OPUC $\varphi_n$, orthogonal in $\T$ (in an hermitian sense) with respect to the weight \eqref{GJOPUC}, also satisfies a full system of non-hermitian (and varying with the degree $n$) orthogonality relations, \[\varphi_n \perp \varphi_m, \qquad 0 \leq m<n\,.\]
Indeed there exists a function $\lambda$ on $\T$, which is constant on each arc defined on $\T$ by the points $\{a_1,\dots,a_m,-1\}$, such that this non-hermitian orthogonality can be expressed in the form
\[\int_\T z^k \varphi_n(z)\frac{W_n(z)}{\lambda(z)}dz=0\,,\qquad k=0,\dots,n-1\,.\]
\end{lemma}
\begin{proof}
Since for $k=0,\dots,n-1$,
\[0=\int_\T z^{-k}\varphi_n(z)\frac{w(z)}{z}dz=\int_\T z^{n-1-k}\varphi_n(z)\, \frac{w(z)}{z^{n}}dz\,,\]
and $n-1-k\in\{0,1,\dots,n-1\}$,
we only need to verify that $z^nW_n(z)/w(z)$ is a constant function in each of the arcs of $\T$. Then it is sufficient to check that
\begin{equation}\label{CocCte}
\frac{z^{-\alpha_j}(z-a_j)^{2\alpha_j}}{|z-a_j|^{2\alpha_j}}\,,
\end{equation}
is constant in the arc which goes from $-1$ to $a_j$ counterclockwise, and is also constant in the arc going from $a_j$ to $-1$ counterclockwise too. Obviously, the modulus of the quotient \eqref{CocCte} is $1$, let us see its argument. The identity
\[\frac{e^{i\theta}-e^{i\tau}}{\exp\left(i\,\frac{\theta+\tau}{2}\right)}
=\exp\left(i\,\frac{\theta-\tau}{2}\right)-\exp\left(i\,\frac{-\theta+\tau}{2}\right)
=2i\sin\frac{\theta-\tau}{2}\in\mathbb{R}i\]
with $\theta,\tau\in(-\pi,\pi]$, shows that with $z=e^{i\theta}$ and $a_j=e^{i\tau}$,
\[\arg(z-a_j)=\frac{\pi}{2}+\frac{\arg(z)+\arg(a_j)}{2}+k_1\pi\,,\]
for some $k_1\in\mathbb{Z}$ and with $\arg(z)$, $\arg(a_j)\in(-\pi,\pi]$. Taking into account the branch cut of $(z-a_j)^{2\alpha_j}$, we choose the argument of $(z-a_j)$ in such a way that
\[\arg(z-\xi)\in \left(\frac{\arg(a_j)}{2}-\pi,\frac{\arg(a_j)}{2}+\pi\right)\,.\]
Notice that $\arg(-1\pm 0i-a_j)=\arg(a_j)/2\pm \pi$, so for $z$ moving along the arc joining $-1$ and $a_j$ counterclockwise, $\arg(z-a_j)$ goes from $\arg(a_j)/2-\pi$ to $\arg(a_j)-\pi/2$, and for $z$ moving in the arc joining $a_j$ and $-1$ counterclockwise, $\arg(z-a_j)$ goes from $\arg(a_j)+\pi/2$ to $\arg(a_j)/2+\pi/2$. Hence, the value of $k_1$ depends on the side where $z$ is in $\T$ with respect to the points $a_j$ and $-1$ (indeed $k_1=-1$ if $\arg(z)\in(-\pi,\arg(a_j))$ and $k_1=0$ if $\arg(z)\in(\arg(a_j),\pi)$). Thus,
\begin{align*}
\arg\left(\frac{z^{-\alpha_j}(z-a_j)^{2\alpha_j}}{|z-a_j|^{2\alpha_j}}\right)
&=-\alpha_j\arg(z)+2\alpha_j\left(\frac{\pi}{2}+\frac{\arg(z)+\arg(a_j)}{2}+k_1\pi\right)+2k_2\pi\,,\\
&=\alpha_j\pi+\alpha_j\arg(a_j)+2\alpha_jk_1\pi+2k_2\pi
\end{align*}
for some $k_1,k_2\in\mathbb{Z}$. These arguments depend only on the arc where $z$ is located, but not on the exact value of $z$, so the quotient \eqref{CocCte} is constant in the arc from $-1$ to $a_j$ and also in the arc from $a_j$ to $-1$. then, we have that
\begin{equation}\label{IdentPesos}
\frac{w(z)}{z^{n}}=\frac{W_n(z)}{\lambda(z)}\,,\qquad z\in\T\,,
\end{equation}
with $\lambda:\T\longrightarrow \T$ being some piecewise constant function, and this proves the lemma.
\end{proof}

Now,
\begin{equation}\label{WnABn}\frac{W'_n(z)}{W_n(z)} = \,-\frac{n+\alpha}{z}\,+\,2\sum_{j=1}^m\,\frac{\alpha_j}{z-a_j} = \,\frac{B_n(z)}{A(z)}\,,
\end{equation}
where $\alpha = \sum_{j=1}^m\,\alpha_j$, $A(z) = z\,\prod_{j=1}^m\,(z-a_j)$ has degree $m+1$, and $B_n$ has degree $m$. This implies that for each $n$, $W_n$ is a semi-classical weight of class $\sigma = m-1$.

Thus, using again the results in \cite[Section 4]{MOS2023}, we conclude that there exists a polynomial $S_n$, of degree at most $m-1$ (the electrostatic partner of $\varphi_n$), such that the second order linear differential \eqref{SODE} equation is satisfied by $y= \varphi_n$, that is,
$$A S_n y'' + ((A'+B) S_n -A S'_n) y'+ C_n y = 0\,,$$
for some polynomial $C_n$.

Therefore, assuming that the zeros of $y= \varphi_n$ are simple, and denoting them by $z_1,\ldots,z_n$, and taking into account that
$$\frac{y''(z_k)}{2 y'(z_k)}\,=\,\sum_{j \neq k}\,\frac{1}{z_k-z_j}\,,$$
\eqref{SODE} yields for any $k=1,\ldots,n$ that
\begin{equation}\label{equilJgral}
\sum_{j \neq k}\,\frac{1}{z_k-z_j}\,+\,\left(\frac{(-n-\alpha +1)/2}{z_k}\,+\,\sum_{j=1}^m\,\frac{\alpha_j +1/2}{z_k - a_j}\,-\,\frac{1}{2}\,\frac{S'_n(z_k)}{S_n(z_k)}\right) = 0\,.
\end{equation}

Thus, we have the following electrostatic model for the zeros of $\varphi_n$.
\begin{theorem}\label{thm:gralized}
Assume that the zeros of $\varphi_n$, the OPUC with respect to weight \eqref{GJOPUC}, are simple. Then, they are in a (possibly unstable) equilibrium in the external field due to $m$ repellents placed at each $a_j$, of respective magnitude $\alpha_j + 1/2$, an attractor of magnitude $(n+\alpha-1)/2$, placed at the origin, as well as at most $m-1$ (counting multiplicities) attractors of magnitude $1/2$ located at the zeros of the corresponding electrostatic partner $S_n$.

As a consequence, the zeros of the OPA $p_n$ (the same assumption about their simplicity is considered) are in equilibrium in the external field created by $m$ repellents at $a_j$ with magnitude $\alpha_j+1/2$, an attractor at the origin of magnitude $(n+\alpha)/2$ and $m-1$ (counting multiplicities) attractors at the roots of the reverse of the electrostatic partner of magnitude $1/2$.
\end{theorem}

In other words, we know the exact behaviour of the roots of $\varphi_n$ except for a few, at most $m-1$, with $m$ independent of $n$, degrees of freedom.
Hence, our main concern is to figure out properties of those zeros of $S_n$.

For our purposes, an interesting subclass of the weight functions \eqref{GJOPUC} consists of the real-symmetric generalized Jacobi weights of the form
\begin{equation}\label{symmGJOPUC}
w(z) = \prod_{j=1}^k\,|z-a_j|^{2\alpha_j}\,|z- \overline{a}_j|^{2\alpha_j},\; a_j\in \T,\,\alpha_j \in \N, \,j=1,\ldots,m\,,
\end{equation}
that is, the function $f$ related to $w$ is a polynomial with $m=2k$ zeros of possibly different multiplicities and real coefficients.

Before we proceed, we are going to provide an expression of the OPA when $f$ is a general polynomial.
\begin{theorem}\label{closedopa}
Let $d, n \in \N$, and $f \in \P_{d},$ with $f(0)=1$. Denote by $p_n$ the OPA to $1/f$ in $H^2$. Then there exist $u_0,...,u_{d-1}$ and $u_{n+d+1},...,u_{n+2d}$ such that
\[p_n(z) = \frac{z^d + \sum_{i=0}^{d-1} u_i z^i + \sum_{j=n+d+1}^{n+2d} u_jz^j}{f(z) f^*(z)}.\]
\end{theorem}
\begin{remark} Notice the following consequences of this result.
    \begin{itemize}
        \item[(a)] The fact that $p_n$ is a polynomial makes the unique choice of values $u_i$  so that the division is exact. This is an advance in the techniques generally known for OPA since it requires solving a square linear system of order $2d$ (rather than $n$) providing a closed formula for each $f$ yielding the OPA of any degree.
        \item[(b)] The formula simplifies when $f$ has real coefficients and even more when the zeros of $f$ are pairs of conjugate points in the Unit Circle, since in that case $f^*\equiv f$ and therefore, all $u_i$ must be real.
        \item[(c)] The condition that $f(0)=1$ does not suppose any lost generality since when $f(0)=0$ all OPA are identically null and otherwise one can divide $f$ by $f(0)$ resulting in multiplying $p_n$ by the same constant.
        \end{itemize}
\end{remark}

We leave the proof of Theorem \ref{closedopa} to a separate subsection (see Subsection 4.2 below).

Observe that from Theorem \ref{closedopa} it is also possible to obtain an analogous formula for the OPUC $\varphi_n$, since it is the reversed polynomial of $p_n$. Namely,
\begin{equation}\label{ExprOPUC}
\varphi_n(z)=\frac{\sum_{i=0}^{d-1}\overline{u_i}z^{n+2d-i}+z^{n+d}+\sum_{j=n+d+1}^{n+2d}\overline{u_j}z^{n+2d-j}}{f(z)f^*(z)}\,.
\end{equation}
Notice that the degree $d$ of $f$ here coincides with the corresponding value of $\alpha=\sum_{j=1}^m \alpha_j$ in the particular case mentioned above. Furthermore, recall that the second kind function associated to $\varphi_n$ is given by
$$\widehat{\varphi_n}(z)=\int_\T\frac{\varphi_n(t)}{t-z}\frac{W_n(t)}{\lambda(t)}dt\,,\qquad z\in\mathbb{C}\setminus{\T}\,,$$
and as consequence the electrostatic partner $S_n$ of $\varphi_n$ can be obtained from these coefficients $u_j$ as the following corollary shows.
\begin{corollary}\label{CorFun2kind}
Let $d, n \in \N$, and $f \in \P_{d}$. Then, the second kind function can be expressed for $|z|>1$ as
$$\widehat{\varphi_n}(z)=-2\pi i \sum_{j=1}^{d}\overline{u_{n+d+j}}z^{-n-j}$$
\end{corollary}
\begin{proof}Consider $t\in\T$ and $z$ with $|z|>1$. Taking into account \eqref{IdentPesos}, the identity
$$f(t)f^*(t)=f(t)\, t^d\, \overline{\,f(1/\overline{t})\,}=f(t)\,t^d\,\overline{\,f(t)\,}=t^d w(t)\,,$$
and Theorem \eqref{closedopa},
the second kind function can be written as
\begin{align*}
\widehat{\varphi_n}(z)&=\int_\T\frac{\varphi_n(t)f(t)f^*(t)}{t-z}\frac{dt}{t^{n+d}}
=\int_\T\frac{\sum_{i=0}^{d-1}\overline{u_i}t^{d-i}+1+\sum_{j=n+d+1}^{n+2d}\overline{u_j}t^{d-j}}{t-z}dt\\
&=\sum_{j=n+d+1}^{n+2d}\overline{u_j}\int_\T\frac{t^{d-j}}{t-z}dt
=\sum_{j=n+d+1}^{n+2d}\overline{u_j}\,(-2\pi i)\,z^{d-j}\,,
\end{align*}
where in the third identity we have used that some of the integrands are analytic and, hence, the corresponding integrals vanish.
\end{proof}
Now, the electrostatic partner can also be computed using \eqref{EP}, \eqref{ExprOPUC} and corollary \ref{CorFun2kind}
\begin{equation}\label{ComputS}
S_n(z)=c\,\left(A(z)\,\varphi_n(z)\,(\widehat{\varphi_n})'(z)-A(z)\,\varphi_n'(z)\,\widehat{\varphi_n}(z)-B_n(z)\,\varphi_n(z)\,\widehat{\varphi_n}(z)\right)
\end{equation}
where $A$ and $B_n$ are given by \eqref{WnABn} and $c$ is a normalizing constant so that $S_n$ is a monic polynomial. Observe that $S_n$ is a polynomial of degree $m-1$, so when we compute it by \eqref{ComputS}, all the coefficients associated to negative powers of $z$ vanish. An alternative way for the computation is
\begin{equation*}S_n(z)=c\frac{A(z)}{z f^2(z)}
\big(z(\varphi_nf^2)(z)\,(\widehat{\varphi_n})'(z)-z(\varphi_nf^2)'(z)\,\widehat{\varphi_n}(z)
+(n+2d)(\varphi_nf^2)(z)\,\widehat{\varphi_n}(z)\big)\,.
\end{equation*}

Next, the simple but very illustrative case when $f$ has two simple symmetric zeros will be analyzed in depth.

\subsection{Two conjugate simple zeros.}\label{sect33}

For the remainder of the Section we will focus on the simple, but non-trivial, case where $k = 1$ and
\begin{equation}\label{ftheta}
f(z) = (z-e^{i\theta})(z-e^{-i\theta})=z^2-2\cos \theta z+1\,.
\end{equation}

This gives rise to the hermitian weight
\begin{equation}\label{hermit3}
w(z) = |z^2-2\cos \theta z+1|^2
\end{equation}
and to the corresponding non-hermitian (and varying) one,
\begin{equation}\label{nonhermit3}
W_n(z) = z^{-n-2}\,(z^2-2\cos \theta z+1)^2.
\end{equation}
Thus, as in previous examples, zeros of $\varphi_n$ are subject to the strong attraction of the origin, to the weak repulsion of unit charges placed at $z=e^{i\theta}$ and $z=e^{-i\theta}$, and another weak attraction: As in Section 3.2, the class of the semi-classical weight is $\sigma = 1$ and, thus, it is generated a weak attraction of magnitude $1/2$ of a charge located at certain point $s_n \in \R$ (possibly depending on the degree $n$).

In this case, \eqref{equilJgral} implies the following electrostatic type identity for the zeros $\{z_k\}$ of $\varphi_n$ and the point $s_n$,
$$\sum_{j \neq k}\,\frac{1}{z_k-z_j}\,+\,\left(\frac{-(n+1)/2}{z_k}\,+\,\frac{3/2}{z_k-e^{i \theta}}\,+\,\frac{3/2}{z_k-e^{-i \theta}}\,-\frac{1}{z_k-s_n}\right)\,=0\,, \qquad k=1,...,n.$$

Now, in order to compute the OPUC $\varphi_n$ we can use the following expression, easily deduced from Theorem \ref{closedopa}:
\begin{equation}\label{expr}
\varphi_n (z) = c \,\frac{t_0+t_1 z+ z^{n+2} + t_2 z^{n+3} + t_3 z^{n+4}}{f(z)^2}\,,
\end{equation}
where $c$ is a constant (that actually depends on $n$), $f(z)$ is given by \eqref{ftheta} and the coefficients $t_i, i=1,\ldots,4$ may be determined by imposing that $\varphi_n$ is indeed a polynomial. This is easier to do than it seems, simply requiring that the numerator has double zeros at $z= e^{\pm i \theta}$. This yields the following linear system for $t_i, i=0,\ldots,3\,,$
\begin{equation}\label{system}
\begin{pmatrix} 1 & e^{i\theta}  &  e^{i (n+3) \theta} &  e^{i (n+4) \theta} \\
 1 &  e^{-i\theta} &  e^{-i (n+3) \theta} & e^{-i (n+4) \theta}  \\
 0 & 1 & (n+3) e^{i (n+2) \theta}  &
 (n+4) e^{i (n+3) \theta}  \\
 0 & 1 &  (n+3) e^{-i (n+2) \theta} &  (n+4) e^{-i (n+3) \theta}
\end{pmatrix} \begin{pmatrix} t_0 \\ t_1 \\ t_2 \\ t_3 \end{pmatrix} =\begin{pmatrix}
- e^{i (n+2) \theta} \\
 - e^{-i (n+2) \theta} \\
 - (n+2) e^{i (n+1) \theta} \\
 - (n+2) e^{-i (n+1) \theta}\,
\end{pmatrix}.
\end{equation}
In the particular case that $\displaystyle \theta = \frac{k \pi}{l}$, with $k,l$ nonnegative and relatively prime integers such that $2k\leq l$, it is easy to see that the only dependence on $\theta$ is through the orbit of $e^{i k \theta}$, giving the same exact unimodular numbers to determine $\varphi_n$ and $\varphi_{n'}$ provided that $n=n' \mod l$.

Denote by $R_0,...,R_3$ the 4 rows of the augmented system and by $\displaystyle \mathcal{S}_{k,l}= \frac{\sin(k\theta)}{\sin(l\theta)}$ (where we are assuming that the denominator is not 0). Then we make some substitutions: $\displaystyle R_0 \to \frac{R_0 + R_1}{2}$, $\displaystyle R_1 \to \frac{R_0 - R_1}{2 i \sin(\theta)}$, $\displaystyle R_2 \to \frac{R_2 + R_3}{2}$ and $\displaystyle R_3 \to \frac{R_2 - R_3}{2 i (n+3) \sin( (n+2)\theta)}$. Thus, we obtain the following augmented matrix for a system equivalent to \eqref{system}:
\begin{equation}\label{system2}
\left(\begin{array}{cccc|c} 1 & \cos(\theta)  &  \cos((n+3) \theta) &  \cos((n+4) \theta) & -\cos((n+2) \theta) \\
 0 &  1 &  \mathcal{S}_{n+3,1} & \mathcal{S}_{n+4,1} &
-\mathcal{S}_{n+2,1}  \\
 0 & 1 & (n+3) \cos((n+2) \theta)  &
 (n+4) \cos((n+3) \theta) & -(n+2) \cos((n+1) \theta)  \\
 0 & 0 &  1 &  \frac{(n+4)}{(n+3)}\mathcal{S}_{n+3,n+2} &  -\frac{(n+2)}{(n+3)} \mathcal{S}_{n+1,n+2}\,
\end{array}\right)
\end{equation}
Part of the relevance of \eqref{system2} lies on the fact that the coefficients are all real, and thus it is easy to check that the solution is real too.

Furthermore, the case of $f(z) = (z^2 - 2 \cos (\theta) z +1)$ was also studied in \cite{AcSe}. Denote by $T_{n,\theta}= \frac{1-e^{2i\theta (n+3)}}{1-e^{2i\theta}}$. Notice that this remains bounded as $n$ varies (its orbit is contained within a circle passing by 0, 1 and $1+e^{2i \theta}$).
In the mentioned article, it was determined that \[1-p_nf(z) =\sum_{j=0}^{n+2} d_{j,n}z^j,\] with \[d_{j,n} = \frac{((n+3)-T_{n,\theta}) e^{-ij\theta} + ((n+3)-\overline{T_{n,\theta}}) e^{ij\theta}}{(n+3)^2-|T_{n,\theta}|^2}.\]

If $\theta = k\pi /l$, as mentioned above, $T_{n,\theta}$ is also periodic, repeating every $l$ values, taking $l$ equidistributed values in the mentioned circle. We can study independently each of the congruent cases. For instance, when $n\equiv -3 \mod l$, then $T_{n,\theta}=0$ and
\[d_{j,n} = \frac{2\cos (j\theta)}{n+3}.\]
This can be used to determine the coefficients of $p_n f$ and those of $p_n f^2$ (or $\varphi_n f^2$) constructively. Indeed it allows to obtain the following asymptotics as $n\to\infty$.

\begin{proposition}\label{PropAstys}
Let us take $\theta=k\pi/ \ell\in(0,\pi/2]$ and for $n\in\mathbb{N}$ consider $n'\in\{0,\dots,2l-1\}$ with $n\equiv n'$ mod $2l$. Then the following asymptotics for the coefficients $t_i$, $i=0,\dots,3$ hold
\begin{align*}
&t_3=1-\frac{2}{n}+O(n^{-2})\,, &&t_2=-2\cos(\theta)+\frac{2\cos(\theta)}{n}+O(n^{-2})\,,\\ &t_1=\frac{2\cos((n'+3)\theta)}{n}+O(n^{-2})\,, &&t_0=-\frac{2\cos((n'+2)\theta)}{n}+O(n^{-2})\,.
\end{align*}
As consequence, if $\cos((n'+3)\theta)\neq 0$, then the spurious zero has the asymptotics
$$s_{n}=\frac{\cos(n'+2)\theta}{\cos(n'+3)\theta}+O(n^{-1})\,,$$
and if $\cos((n'+3)\theta)=0$ then the ghost charge diverges to $\infty$.
\end{proposition}
\begin{proof}
Let us consider the system \eqref{system} with the substitution $n\mapsto n'$ just in the exponentials. Applying Cramer's rule we can conclude that coefficients $t_i$, $i=0,\dots,3$ are rational functions on the variable $n$ (these coefficients depend on $n'$ and $\theta$ too). Indeed the numerators of such rational functions has degree at most $2$ in $n$ while the denominator is a polynomial of degree exactly two at least for $n$ large enough since its coefficient in $n^2$ is
$$2\cos(2\theta)-2\,,$$
which does not vanish since $\theta\neq 0$. Hence, the following type of asymptotics for the sequences $\{t_i\}$ holds
$$t_i=t_{i,0}+\frac{t_{i,1}}{n}+O(n^{-2})\,,$$
where the coefficients $t_{i,0}$ and $t_{i,1}$ may depend on $n'$ and $\theta$. In fact if we divide the third and fourth rows of the previous system by $n$ and take limit as $n\to\infty$ through a subsequence $\{n'+r2l\}_{r\in\mathbb{N}}$ we obtain the following system for $t_{i,0}$, $i=0,\dots,3$ $$
\begin{pmatrix} 1 & e^{i\theta}  &  e^{i (n'+3) \theta} &  e^{i (n'+4) \theta} \\
 1 &  e^{-i\theta} &  e^{-i (n'+3) \theta} & e^{-i (n'+4) \theta} \\
 0 & 0  & e^{i (n'+2) \theta}  & e^{i (n'+3) \theta}  \\
 0 & 0 &  e^{-i (n'+2) \theta} &  e^{-i (n'+3) \theta}
\end{pmatrix}
\begin{pmatrix}t_{0,0}\\t_{1,0}\\t_{2,0}\\t_{3,0}\end{pmatrix}
=
\begin{pmatrix}
-e^{i(n'+2)\theta}\\-e^{-i(n'+2)\theta} \\-e^{i(n'+1)\theta}\\-e^{-i(n'+1)\theta}
\end{pmatrix}
$$
whose unique solution is
$$ t_{3,0}=1\,,\qquad t_{2,0}=-2\cos\theta\,,\qquad t_{1,0}=0\,,\qquad t_{0,0}=0\,.$$
With the standard techniques and the same ideas we can obtain the following system for $t_{i,1}$, $i=0,\dots,3$
$$
\begin{pmatrix} 1 & e^{i\theta}  &  e^{i (n'+3) \theta} &  e^{i (n'+4) \theta} \\
 1 &  e^{-i\theta} &  e^{-i (n'+3) \theta} & e^{-i (n'+4) \theta} \\
 0 & 0  & e^{i (n'+2) \theta}  & e^{i (n'+3) \theta}  \\
 0 & 0 &  e^{-i (n'+2) \theta} &  e^{-i (n'+3) \theta}
\end{pmatrix}
\begin{pmatrix}t_{0,1}\\t_{1,1}\\t_{2,1}\\t_{3,1}\end{pmatrix}
=
\begin{pmatrix}
0\\0 \\e^{i(n'+1)\theta}-e^{i(n'+3)\theta}\\e^{-i(n'+1)\theta}-e^{-i(n'+3)\theta}
\end{pmatrix}
$$
which has 
$$t_{3,1}=-2\,,\qquad t_{2,1}=2\cos\theta\,,\qquad t_{1,1}=2\cos(n'+3)\theta\,,\qquad t_{0,1}=-2\cos(n'+2)\theta\,,$$
as the unique solution. Then the asymptotics for the coefficients $t_{i}$, $i=0,\dots,3$ holds.

The electrostatic partner $S_n$ can be computed from the identity \eqref{ComputS} and Corollary \ref{CorFun2kind} as
$$S_n(z)=z+\frac{(n+4)t_0}{(n+3)t_1}+\frac{(n+2)t_2}{(n+3)t_3}+2\cos(\theta)\,,$$
from where the asymptotics for the ghost charge $s_n$ is obtained (observe that $s_n$ diverge only when the quotient $t_0/t_1$ diverges  which happens when $t_1=O(n^{-2})$).
\end{proof}

\begin{remark}Notice that the behaviour of the ghost charge is periodic with period $\ell$ instead of $2l$.
\end{remark}

In Figure \ref{Fig2conjugate} we can see an example of the behaviour of the roots of $\varphi_n$, its electrostatic partner, $p_n$, and the reverse of the electrostatic partner. In this example, the coefficients $T_{n,\theta}$ are periodic in such a way that we observe 3 different behaviours depending on $n\equiv 1,2,3$ mod $3$. These three types of behaviours can also be explained by the limit of the ghost root given in Proposition \ref{PropAstys}: $s_n\to1/2$ for $n\equiv 0$ mod $3$, $s_n\to 2$ for $n\equiv 1$ mod $3$ and $s_n\to -1$ for $n\equiv 2$ mod $3$. In these cases, when $n\equiv 0$ modifies $3$, the root of the electrostatic partner is inside the circle and attracts with it a root of $\varphi_n$ (the same behavior as in the examples of Figure \eqref{Fig2_1and-1} with odd $n$).  When $n\equiv 1$ mod $3$, the root of the electrostatic partner is outside the circle, but this does not produce any special phenomenon. When $n\equiv 2$ mod $3$, the root of the electrostatic partner get close to the circle from inside, deforming the contour where the roots of $\varphi_n$ or $p_n$ lie.

\begin{figure}
\includegraphics[width=6cm]{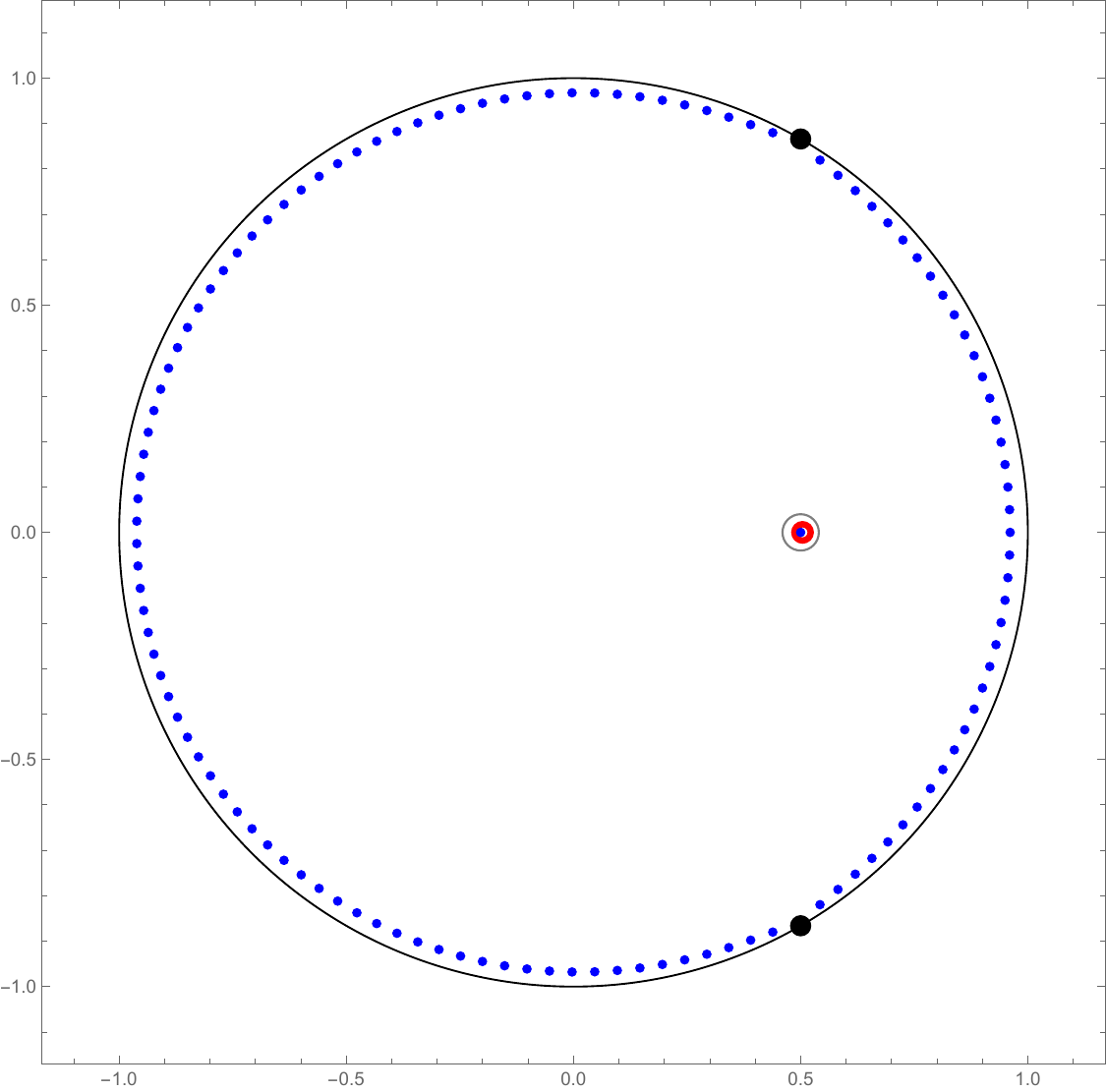}
\includegraphics[width=6cm]{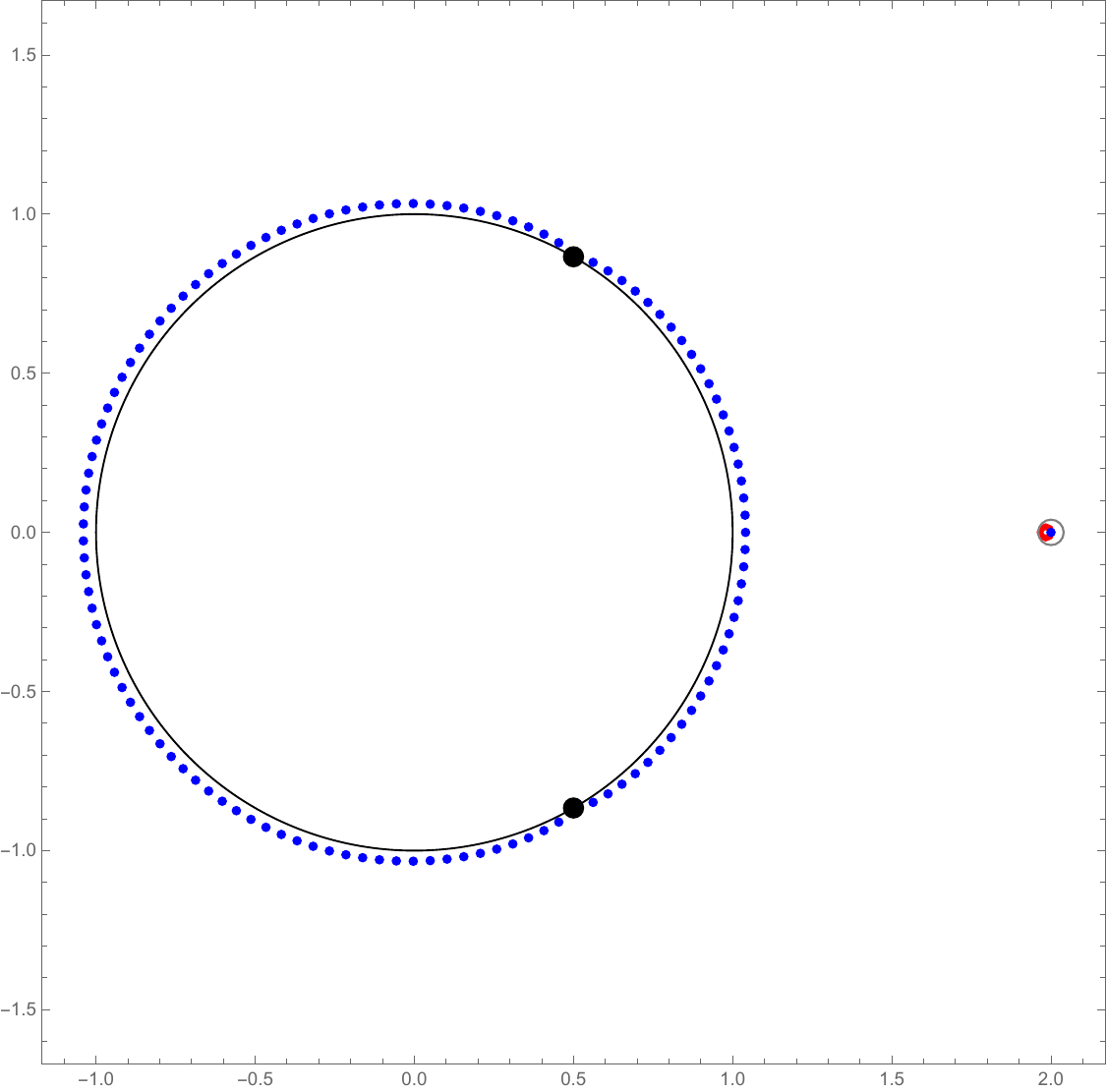}
\includegraphics[width=6cm]{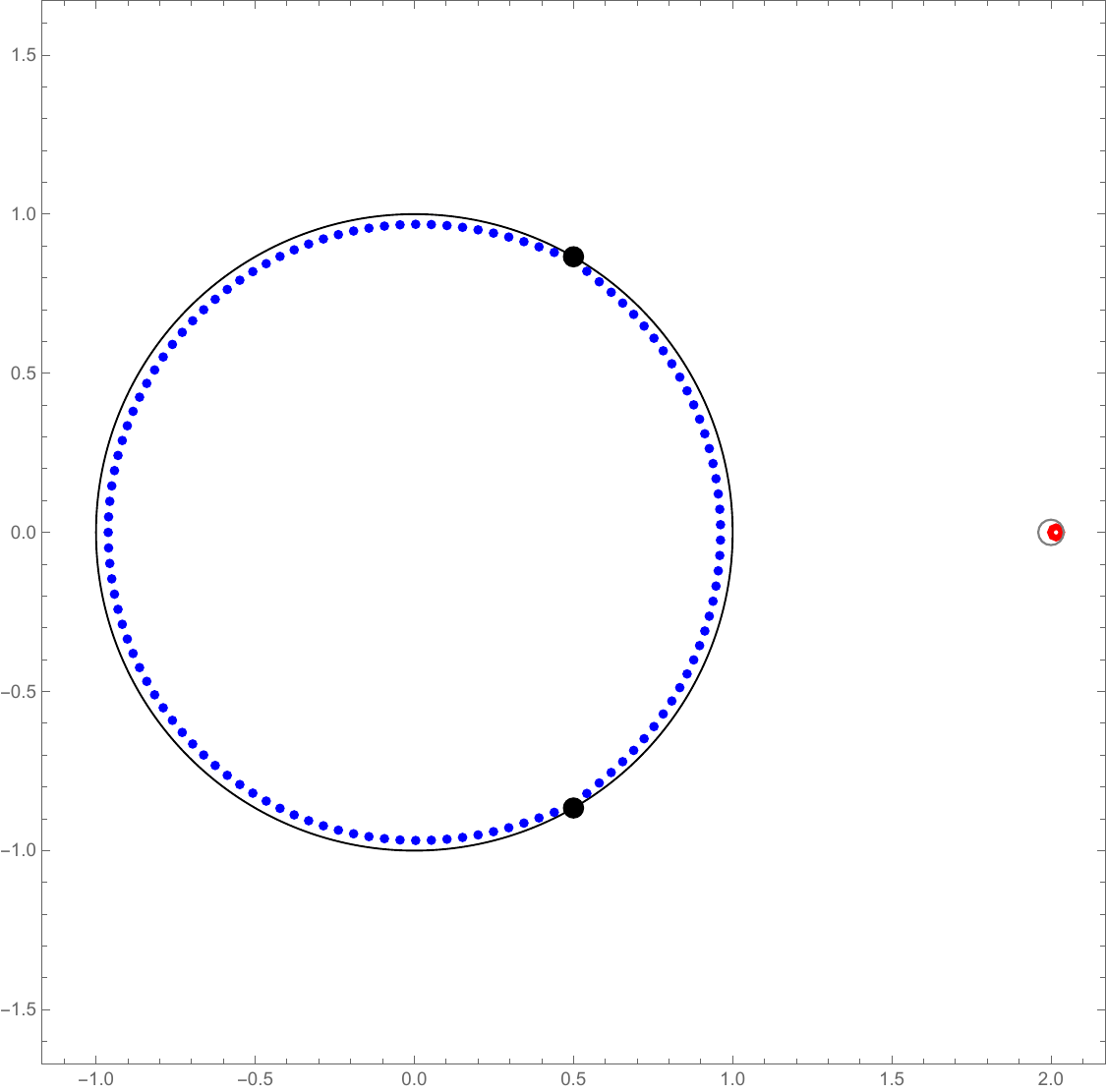}
\includegraphics[width=6cm]{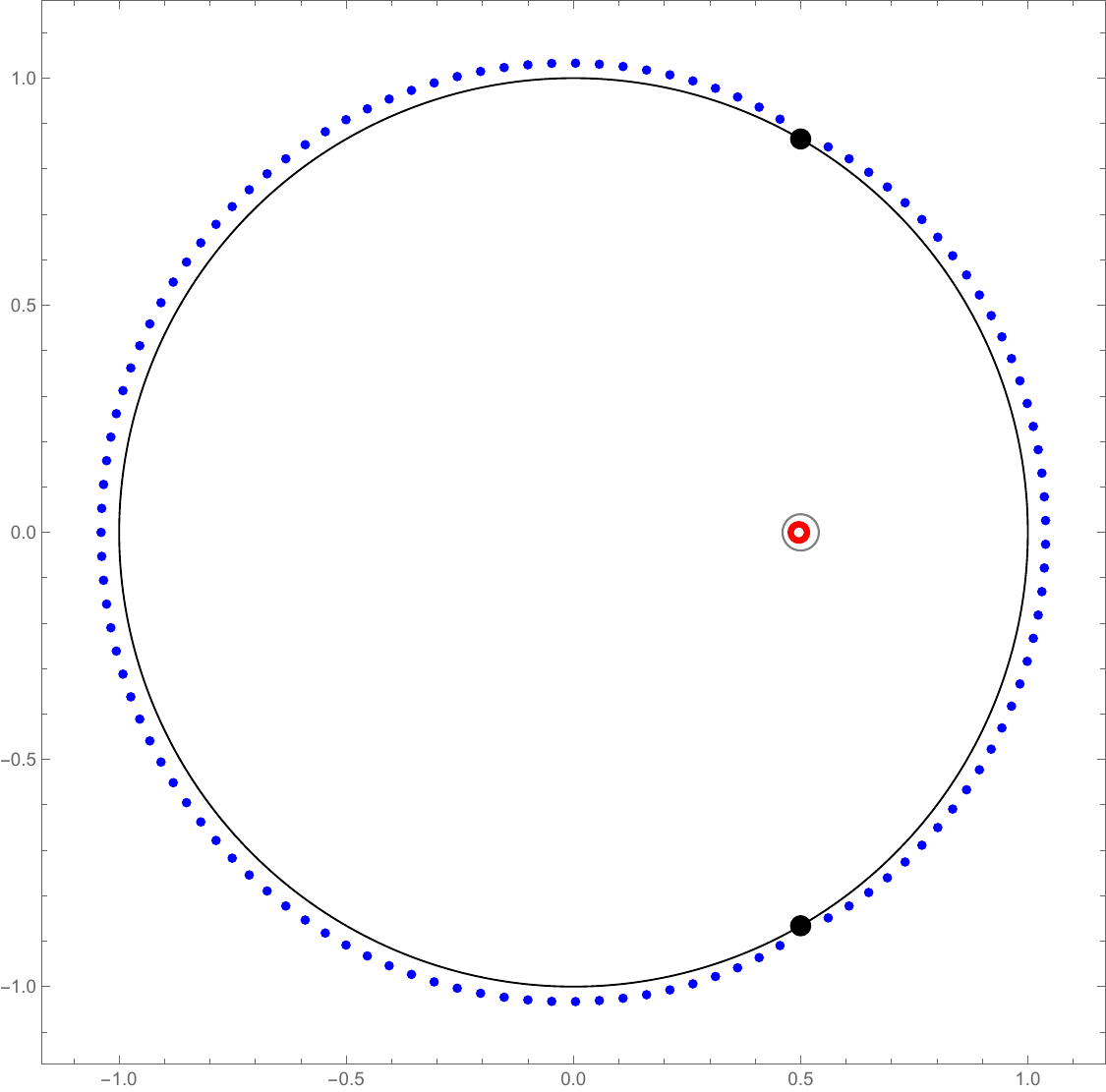}
\includegraphics[width=6cm]{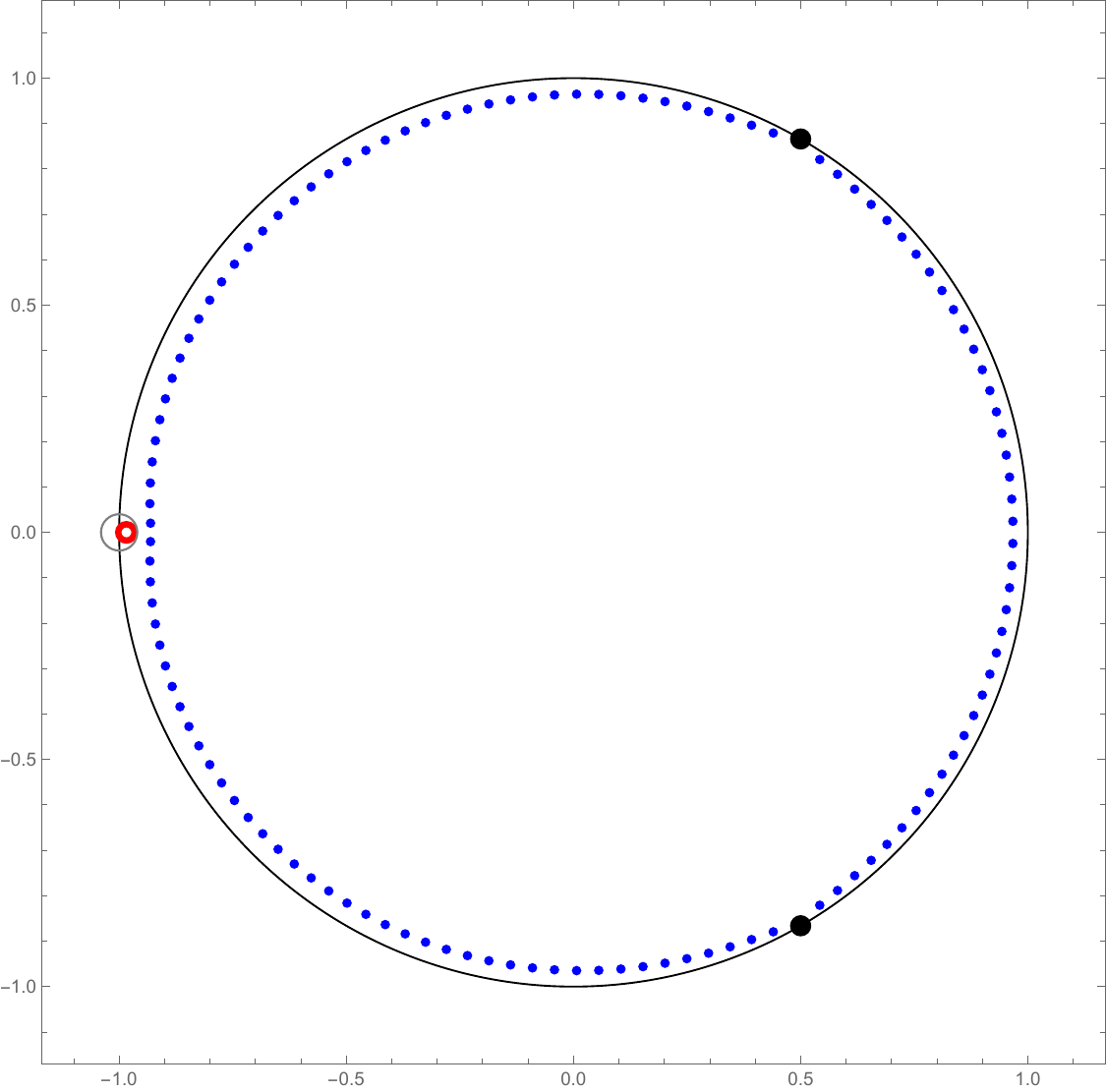}
\includegraphics[width=6cm]{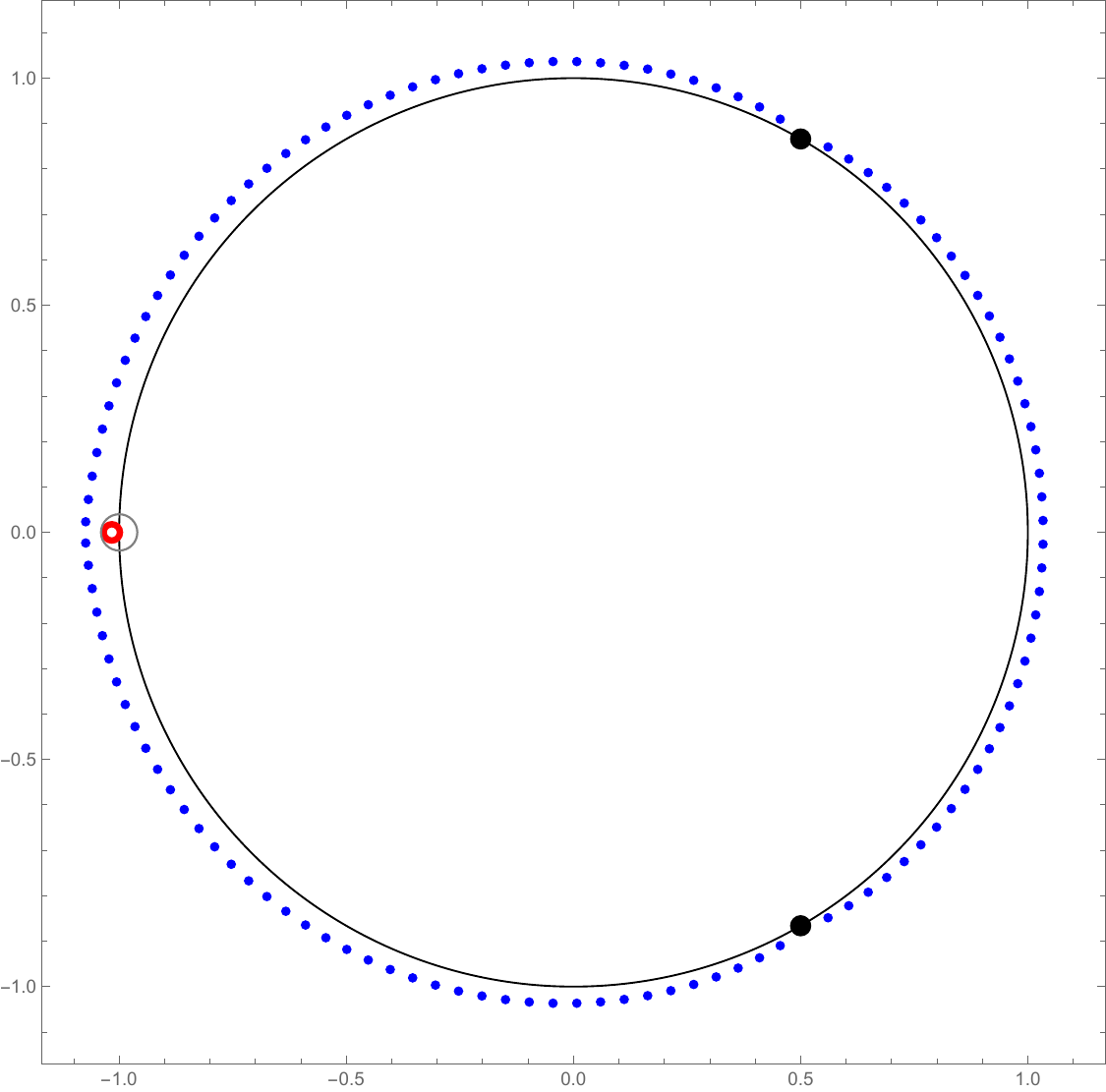}
\caption{\label{Fig2conjugate} Roots (in blue) of $\varphi_n$ (left) and $p_n$ (right) for degrees $n=120\equiv 0$ mod $3$ (first row), $n=121\equiv 1$ mod $3$  (second) and $n=122\equiv 2$ mod $3$ (last) associated to $w(z)= |z-\exp(i\pi/3)|^2\,|z-\exp(-i\pi/3)|^2$ or, equivalently, to $f(z)=(z-\exp(i\pi/3))\,(z-\exp(-i\pi/3))$. The spurious zero (left side) and its reverse conjugate (right side) correspond with the red circle, while the gray and thinner circle is centered at the limit of these spurious roots. Observe the effect of these spurious zeros (ghost charges) when it is inside the circle, attracting one root of $\varphi_n$, and deforming in turn the contour where the zeros of $\varphi_n$ are located.}
\end{figure}

\subsection{A closed formula for many OPA}
Now we turn to the proof of Theorem \ref{closedopa}.

\begin{proof}
First of all, $p_n$ is already known to be a polynomial (of degree at most $n$) and thus $p_n f f^*$ is already assumed to be a polynomial of degree at most $n+2d$. What remains is only to show that the coefficients of $p_n f f^*$ of degrees between $d+1$ and $n+d$ are 0 and that of degree $d$ is equal to 1.

Now, we are going to revisit the main Theorem in \cite{BCLSS2015} but use the form from Theorem 2.1 in \cite{FMS2014}. There it is shown that the coefficients $c= (c_k)_{k=0}^n$ of the OPA to $1/f$ in $H^2$ of degree $n$ are given as the only solution to an invertible $(n+1) \times (n+1)$ linear system of the form $Mc=b$ where
\[b^T=(\overline{f(0)},0,...,0).\] In our case $f(0)=1$ and so, $b=e_0$ (the first element of the canonical basis), whereas $M$ is a matrix with entries $M_{j,k}= \left\langle z^k f, z^j f\right\rangle$. Again, focusing on our case, the shift in $H^2$ is an isometry, thus $M_{j,k}=M_{j+1,k+1}$ (whenever that is well defined) and this means that the matrix $M$ happens to be a Toeplitz matrix (i.e., each diagonal takes a constant value).

The equations in said linear system are of the form \[\sum_{k=0}^n M_{j,k} c_k = 0\] for $j=1,...,n$ and \[\sum_{k=0}^n M_{0,k} c_k = 1.\]

What remains to be mentioned is that $M_{j,k}$ coincides with the coefficient of order $d+j-k$ of $ff^*$. To see this, since the matrix is hermitian and $(ff^*)*$ is a polynomial with coefficients conjugate to those of $ff^*$, we only need to check the case $k \geq j$. By the Toeplitz property, we only need to check that the coefficient of $ff^*$ of order $d-t$, $C_{d-t}$ ($t=0,...,n$) coincides with $M_{0,t}$. Denote $f(z) = \sum_{j=0}^d a_j z^j$ (and also denote $a_j = 0$ if $d<j \leq n$). Then \[M_{0,k}= \left \langle z^k f, f\right \rangle = \sum_{t=k}^n a_{t-k} \overline{a_t}.\]
Notice that from the definition of $f^*$ as the reversed polynomial of $f$, we have \[C_{d-t}= \sum_{s=0}^{d-t} a_s \overline{a_{d-(d-t-s)}}=\sum_{s=0}^{d} a_s \overline{a_{s+t}}.\]
Notice all the values that are null in each of the two results and this implies they are the same, completing  the proof.
\end{proof}

\section{Spurious zeros of OPA and OPUC.}

In the examples displayed in Subsection 4.1, the possible appearance of spurious zeros of the OPUC $\varphi_n$ (and of the OPA $p_n$), which means the existence of ghost charges, was shown, coinciding with what was previously shown in \cite{MMS2006}. Indeed, we saw that though the vast majority of zeros of OPUC (OPA) cluster on $\T$ from inside (respect., from outside), it is possible the existence of a spurious zero of the OPUC inside the open Unit Disk (respect., of the OPA outside the closed Unit Disk). Moreover, it was also shown that the possible appearance of spurious zeros is related to the congruences $n\equiv r, \,\text{mod}\,l $.
In such cases, this spurious zero comes together the unique zero of the electrostatic partner and, vice versa, if the zero of the electrostatic partner, $s_n$, lies inside the open Unit Disk, it attracts a spurious zero of the OPUC (respect., its reversed point outside attracts a spurious zero of the OPA).

The notion of spurious poles (or zeros) was probably coined by G. Baker in 1960's (see \cite{Stahl}) in the context of the Pad\'e Approximants (PA, that is, Taylor rational interpolants to a function), and were systematically studied by H. Stahl \cite{Stahl}. He uses a Cauchy integral representation of the function on a contour $C$ of the complex plane, and proves that just a few of the poles (zeros of a non-hermitian orthogonal polynomial in $C$)  may cluster far from the integration contour; these are the spurious poles (observe, in turn, the asymptotic nature of this concept). Under quite general conditions on $f$, each of these spurious poles is accompanied by a spurious zero (i.e., a zero of the numerator), in such a way that they asymptotically cancel and do not break the convergence far from this small zone; however, the presence of these pairs of spurious poles/zeros makes the uniform convergence in compact subsets of the domain of holomorphy of the function impossible, so that only convergence in capacity takes place. Seemingly, these pairs of spurious poles/zeros were observed quite some time ago by physicists, and were known in the scientific literature as \textit{Froissart doublettes}.
In our case, it seems to be a similar relationship (doublettes) between the zeros of our electrostatic partner $S$ and the spurious zeros of the OPUC $\varphi_n$ (and, after reversing them, with those of the OPA).

That is, as far as the numerical evidence shows, spurious zeros of $\varphi_n$ ($p_n$) and zeros of the electrostatic partner lying inside the Unit Disk (respect., reversed of zeros of the electrostatic partner lying outisde the closed disk) come together for our case study.
In this sense, we propose the following conjecture linking spurious roots of $\varphi_n$ with roots of the electrostatic partner.
\begin{conjecture} Let $\{\xi_{n,k}\}$ be the set of spurious roots of $\varphi_n$ and $\{s_{n,\ell}\}$ the set of the roots of the electrostatic partner $S_n$ which are inside the disk and do not accumulate on $\T$. Then both sets has the same amount of elements and we can reorganize all of them in pairs $(\xi_{n,k},s_{n,k})$ in such a way that
$$\xi_{n,k}=s_{n,k}-\frac{s_{n,k}}{n}+o(1/n)\,,$$
for any of these pairs.
\end{conjecture}
A part of this conjecture can be justified using the electrostatic interpretation. Suppose that $\{z_{n,k}\}_{k=1}^n$ are the roots of $\varphi_n$ and that $\xi_{n,1}$ is indeed $z_{n,1}$. Suppose also that $\xi_{n,1}$ does not converges to $0$. Denote by $\{s_{n,k}\}_{k=1}^{m-1}$ the roots of $S_n$. Then \eqref{equilJgral} gives
 \begin{equation}\label{equilJgralSpur}
\sum_{j=2}^{n}\,\frac{1}{z_{n,1}-z_{n,j}}\,+\,\frac{(-n-\alpha +1)/2}{z_{n,1}}\,+\,\sum_{j=1}^m\,\frac{\alpha_j +1/2}{z_{n,1} - a_j}\,-\sum_{j=1}^{m-1}\,\frac{1/2}{z_{n,1}-s_{n,j}} = 0\,.
\end{equation}
If the rest of the roots $z_{n,2},\dots,z_{n,n}$ do not get close to $z_{n,1}$ as $n\to\infty$, then the first addend in \eqref{equilJgralSpur} is $o(n)$; to see it take into account that though at a first view it would be $O(n)$, the limit distribution of these roots is the uniform distribution on $\T$, whose logarithmic potential vanishes in the interior of the disk $\{z:|z|<1\}$, and so it does the Cauchy transform, which is the limit of this first addend). The second addend is $-n/(2z_{n,1})+O(1)$, and the third addend is $O(1)$. Hence, the fourth term must be $n/z_1+o(n)$, but the only possible way for this to happen is that a root of $S_n$ gets close to $\xi_{n,1}$; let $s_{n,1}$ be this root. Then, \eqref{equilJgralSpur} simplifies
$$\frac{-n/2}{z_{n,1}}-\frac{1/2}{z_{n,1}-s_{n,1}}=o(n)\,,$$
which would provide the relationship asserted in the conjecture.

Figure \ref{FigEspureos} contains information about the roots the OPUC associated to the weight $w(t)=|t-e^{i\pi/3}|^2|t-e^{i2\pi/3}|^2 |t-e^{i-\pi/3}|^2$, with $n\equiv 2$ mod 6. Indeed, we have focused on the the information about spurious zeros. These polynomials present at most two spurious zeros according to \cite{MMS2006}. Indeed, when $n\equiv 2$ mod 6 polynomials $\varphi_n$ present exactly two spurious zeros, which converge quickly to $-0.0227405 + 0.796708 i$ and $0.356074 - 0.219358 i$, approximately (for others $n$'s we have 1 or none spurious roots). There, we can see the relationship between spurious roots and the roots of the electrostatic partner. Moreover, not only both types of roots seem to converge to the same points, but also the terms
$$n\left(\xi_{n,j}+\frac{\xi_{n,j}}{n}-s_{n,j}\right)\,,$$
seem to be $o(1)$, which agrees with our conjecture. In this sense, take into account that it is much easier to predict the appearance of roots of the electrostatic partner inside the circle from \eqref{ComputS} than that of spurious zeros of $\varphi_n$ (or $p_n$).

\begin{figure}
\includegraphics[width=6cm]{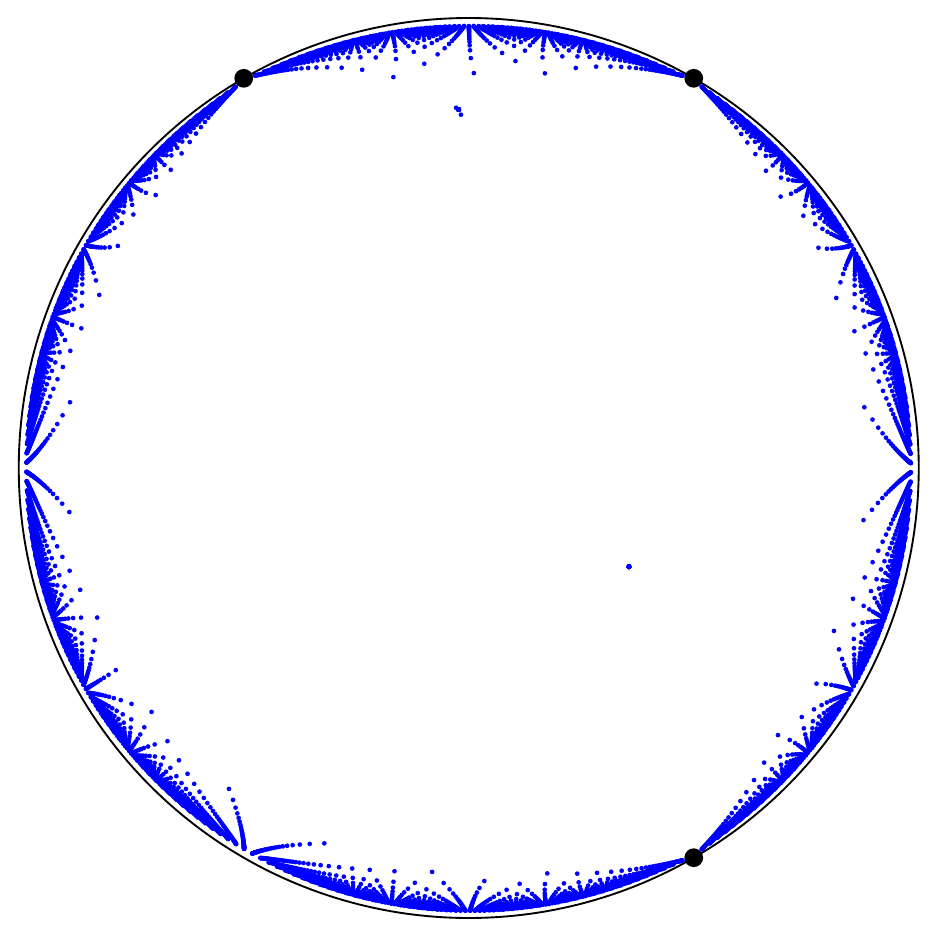}
\includegraphics[width=6cm]{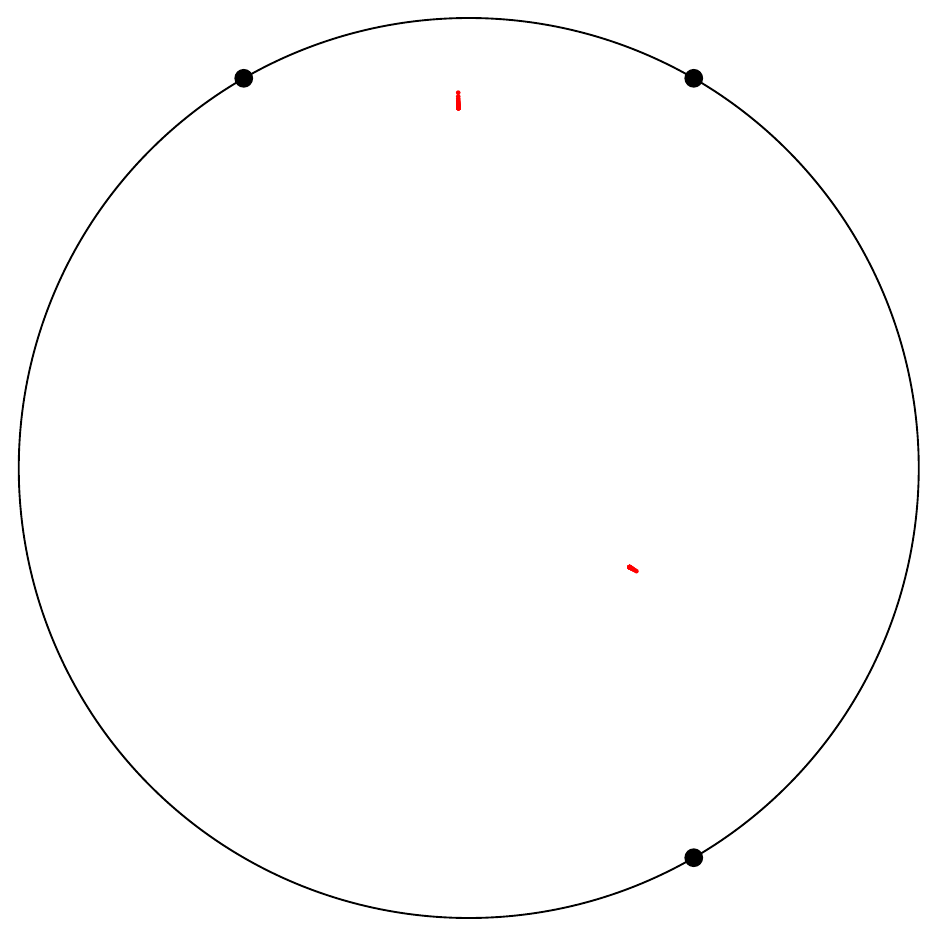}
\includegraphics[width=6cm]{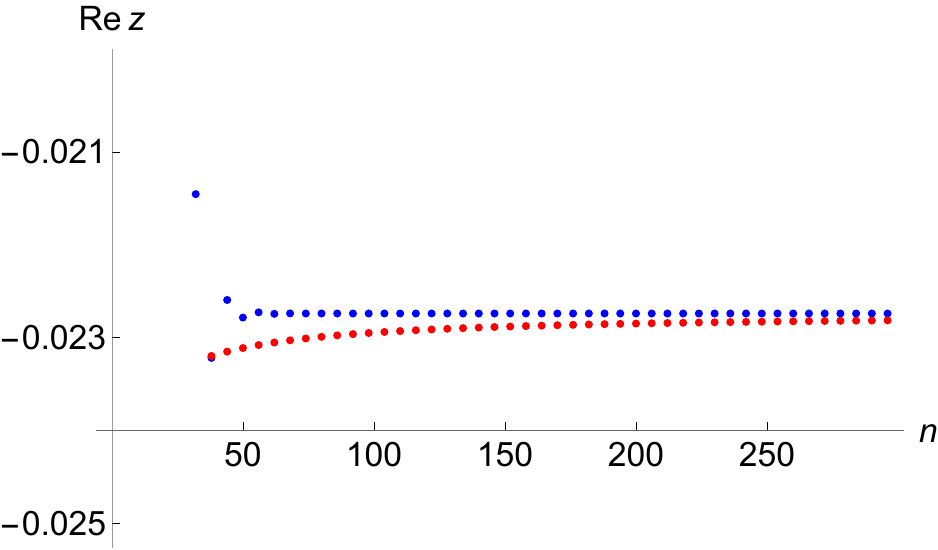}\hspace{10mm}
\includegraphics[width=6cm]{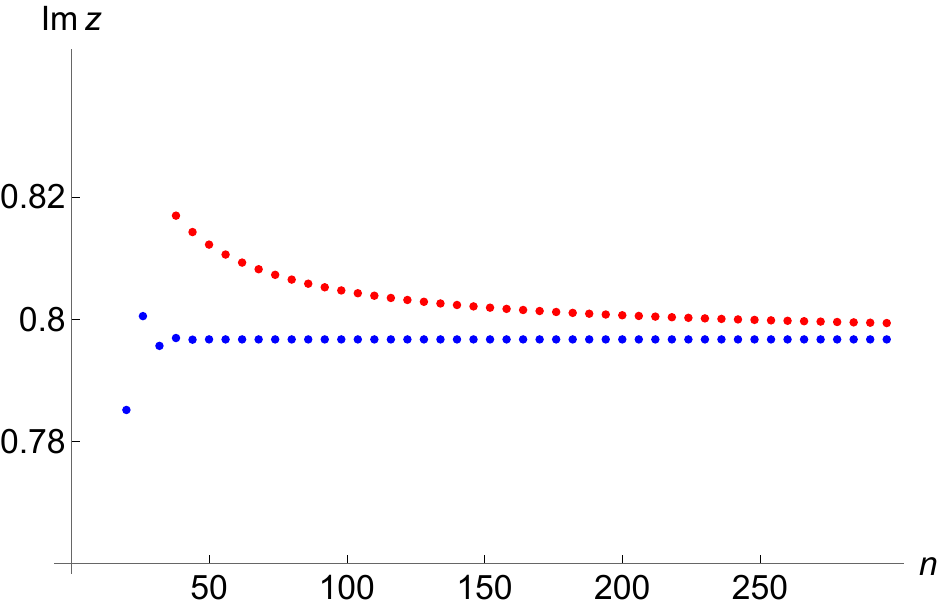}
\includegraphics[width=6cm]{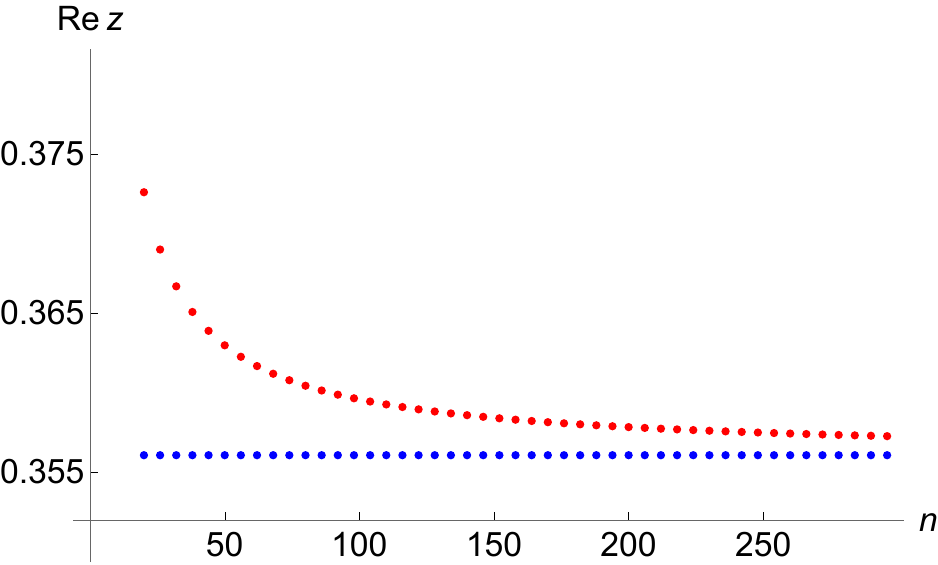}\hspace{10mm}
\includegraphics[width=6cm]{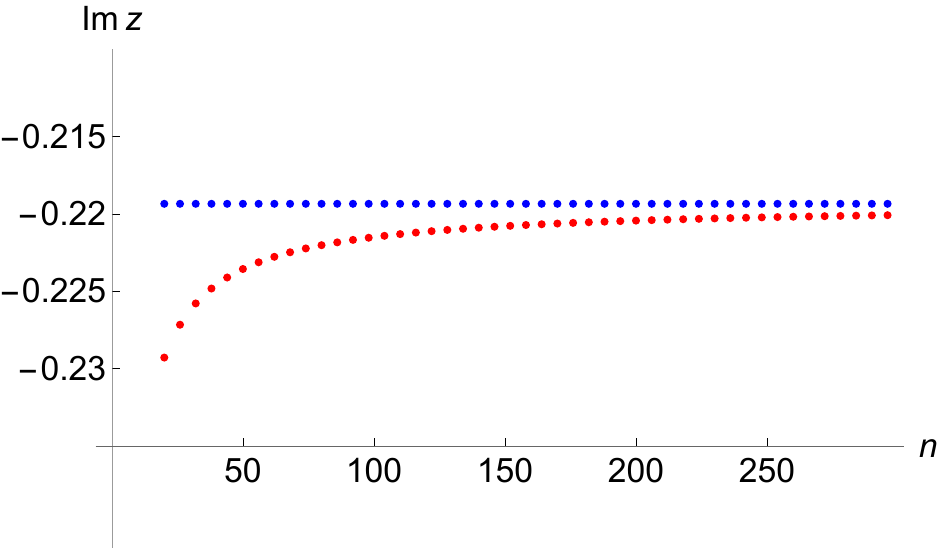}
\includegraphics[width=6cm]{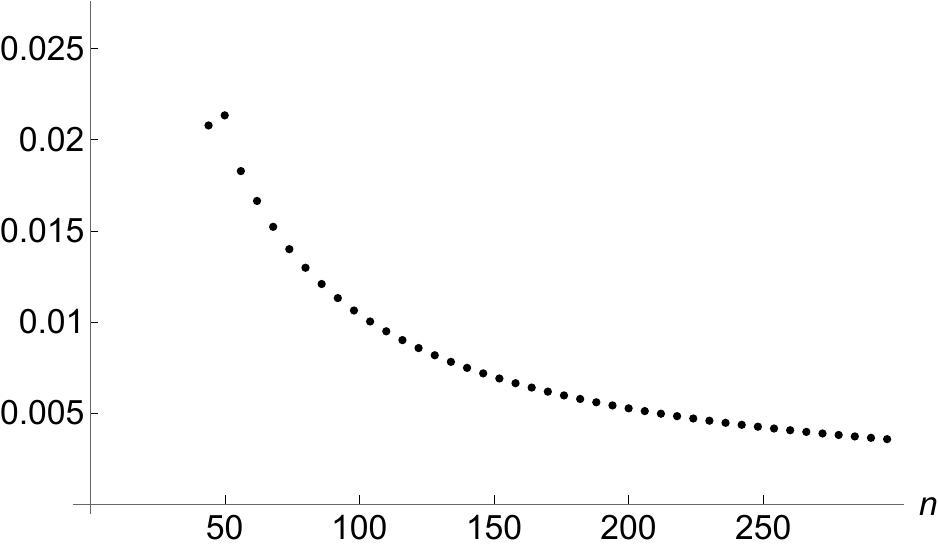}\hspace{10mm}
\includegraphics[width=6cm]{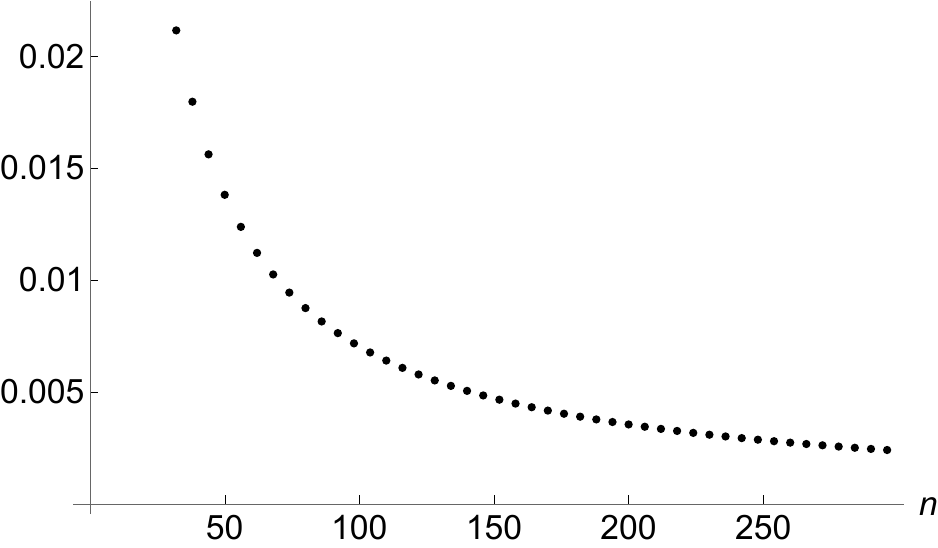}
\caption{\label{FigEspureos} Zeros of OPUC and electrostatic partners for ($w(t)=|t-e^{i\pi/3}|^2|t-e^{i2\pi/3}|^2 |t-e^{i-\pi/3}|^2$, $n\equiv 2$ mod 6) are shown. First row left (resp. right): Roots of $\varphi_n$ for $n\in\{20,22,...,98\}$ (resp. $S_n$). Second left (resp. right): real (resp. imaginary) part of spurious zero (blue) of $\varphi_n$ near  $-0.0227405 + 0.796708 i$, and zero $s_n$ of the associated $S_n$ (red). Third left (resp. right): real (resp. imaginary) part of spurious zero (blue) of $\varphi_n$ near  $0.356074 - 0.219358 i$, and the root of $S_n$ (red). Fourth left (resp. right): $n(\xi_n-\xi_n/n-s_n)$ for $\xi_n$ spurious zero near $-0.0227405 + 0.796708 i$ (resp. $0.356074 - 0.219358 i$) and $s_n$, the zero of the associated electrostatic partner $S_n$.}
\end{figure}

\section{Further problems}\label{sect4}

It would be unnatural to leave at this point without mentioning that this is not the end of the road, and that one should aim at finding more general classes of functions for which a similar approach can be developed. Here we have studied the role in the repulsion or the attraction of the roots of the OPAs, and the related OPUC, played by the zeros of $f$ regarding their position and their multiplicity. One could aim at studying how the product of two different functions $f_1$ and $f_2$ may lead to a similar problem; how to extend the consequences to other spaces like Dirichlet or Bergman; and how to pass from a sequence of functions $\{f_n\}$ for which the behaviour of the corresponding OPA is sufficiently well understood, to infer properties of opas of the limit of said sequence in the spirit of \cite{AgSe}.

Next, we will list some possible problems, or further developments, that we can face now or in the next future.

\begin{itemize}
\item As shown, the numerical evidence and Conjecture 1 above suggest that for the generalized Jacobi type functions $f$ and related weights $W_n$, the zeros of the electrostatic partner which lie inside the Unit Disk and the spurious zeros of the corresponding OPUC $\varphi_n$ are clearly related. We plan to go deep in this connection in forthcoming papers.
\item Even if it is not directly related with the study of cyclicity, it seems reasonable to study a case where $f$ has varying multiplicities, for instance, $f_n(z) = (1-z)^{a_n}$, with, say, $\lim_{n\rightarrow\infty}\,\frac{a_n}{n} = C>0$.
Understanding the asymptotic behavior of the sequence of $p_n$, where for each $n$, $p_n$ is the nth OPA to $1/f_n$ may actually help understand the individual role played by each 0 of $f$.
\item For Bergman-type spaces, one should deal with orthogonality not just on $\T$, but on the whole $\D$, using $dA(z)$, that is, with the so-called Bergman orthogonal polynomials. For Dirichlet, an additional derivative may ruin the nature of the differential equation regarding the OPA. The connection in these spaces with the orthogonal polynomials still holds via the reproducing formula but this does not imply $p_n$ and $\varphi_n$ are reversed from each other anymore. Also, the property used in the last Theorem, that the matrix is Toeplitz is only true in $H^2$.
\item If one is successful with the basic problems for polynomial $f$ with a few zeros, in the Dirichlet space, then a natural direction to follow is that of constructing $f_n$ with sets of zeros  $Z_n$ such that \[\cup_{n\in \N} Z_n = Z\] is an explicit set of logarithmic capacity zero. Perhaps a counterexample to the Brown-Shields conjecture \cite{BrSh} consists of choosing adequately the multiplicities of each zero.
\end{itemize}

\end{document}